\documentclass[a4paper, 12pt, oneside, reqno]{amsart}

\usepackage[top=35mm,bottom=15mm,left=20mm,right=20mm]{geometry}

\usepackage[foot]{amsaddr} 
\usepackage{mathrsfs}
\usepackage{amsfonts}
\usepackage{amssymb}
\usepackage{amsxtra}
\usepackage{color}
\usepackage{dsfont}
\usepackage{bbm}

\usepackage[english,polish]{babel}
\usepackage[colorlinks,citecolor=blue,urlcolor=blue,bookmarks=true]{hyperref}
\usepackage{mathtools}

\newcommand{\R}{\mathbb{R}}
\newcommand{\Ss}{\mathbb{S}}
\newcommand{\Per}{\mathrm{Per}}
\newcommand{\calH}{\mathcal{H}}
\newcommand{\calC}{\mathcal{C}}
\newcommand{\ud}{\mathrm{d}}
\newcommand{\Rd}{\R^d}
\newcommand{\eps}{\varepsilon}
\newcommand{\da}{\downarrow}
\newcommand{\vp}{\varphi}
\newcommand{\fN}{\mathfrak{N}}
\newcommand{\rr}{\mathcal{R}}
\newcommand{\Reg}{\mathfrak{C}}
\newcommand{\inn}{\mathrm{inn}}
\newcommand{\ou}{\mathrm{ou}}
\newcommand{\piu}{\pi^{\uparrow}}
\newcommand{\Ku}{K^{\uparrow}}
\newcommand{\tauhA}{\widehat{A}_{\tau}}
\newcommand{\tauhE}{\widehat{E}_{\tau}}
\newcommand{\tauE}{E_{\tau}}
\newcommand{\tauF}{F_{\tau}}

\newcommand{\pl}[1]{\foreignlanguage{polish}{#1}}

\newtheorem{theorem}{Theorem}[section]
\newtheorem{proposition}[theorem]{Proposition}
\newtheorem{lemma}[theorem]{Lemma}
\newtheorem{corollary}[theorem]{Corollary}

\theoremstyle{definition}

\newtheorem{example}[theorem]{Example}
\newtheorem{remark}[theorem]{Remark}

\title[Nonlocal curvatures]{Asymptotics and geometric flows\\ for a class of nonlocal curvatures}
\author{Wojciech Cygan $^{1,2}$}
\address{$^1$University of Wroc{\l}aw,
		Faculty of Mathematics and Computer Science\\
		Institute of Mathematics,
		pl.\ Grunwaldzki 2/4, 50--384 Wroc{\l}aw, Poland}
\address{$^2$Technische Universit\"{a}t Dresden,
		Faculty of Mathematics\\
		Institute of Mathematical Stochastics,
		Zellescher Weg 25, 01069 Dresden, Germany}
\email{wojciech.cygan@uwr.edu.pl}
\thanks{Research supported by National Science Centre (Poland), grant no.\ 2019/33/B/ST1/02494}

\author{Tomasz Grzywny $^{3}$}
\address{$^{3}$Wroc{\l}aw University of Science and Technology,
Faculty of Pure and Applied Mathematics\\
	Wyb. \pl{Wyspia\'{n}skiego} 27,
	50-370 \pl{Wroc\l{}aw}, Poland}
\email{tomasz.grzywny@pwr.edu.pl}
\author[J.~Lenczewska]{Julia Lenczewska $^{3}$}
\email{julia.lenczewska@pwr.edu.pl}

 				

\numberwithin{equation}{section}
\allowdisplaybreaks

\begin{document}
\selectlanguage{english}

\begin{abstract}
We consider a family of nonlocal curvatures determined through a kernel 
which is symmetric and bounded from above by a radial and radially non-increasing profile satisfying an integrability condition.
It turns out that such definition encompasses various variants of nonlocal curvatures that have already appeared in the literature, including fractional curvature and anisotropic fractional curvature. The main task undertaken in the article is to study the limit behaviour of the introduced nonlocal curvatures under an appropriate limiting procedure. This enables us to recover known asymptotic results e.g.\ for the fractional curvature and for the anisotropic fractional curvature. For the convergence of anisotropic fractional curvatures we identify the limit object as the nonlocal curvature being the first variation of the related anisotropic fractional perimeter. We also prove existence, uniqueness and stability of viscosity solutions to the corresponding level-set parabolic Cauchy problem formulated in terms of the investigated nonlocal curvature. 
\end{abstract}

\maketitle

\section{Introduction}

The present article is concerned with the asymptotic behaviour of a certain class of nonlocal curvatures. 
 We are especially interested in the anisotropic fractional curvature which is given by the first variation of the anisotropic fractional perimeter. 
The main goal of the present work is to establish convergence of anisotropic fractional curvatures, and to identify the limit object.  
 Anisotropic curvatures have been intensively studied over last few decades, see e.g.\ \cite{Bellettini}, \cite{JEMS}, \cite{Cesaroni-Stability}, \cite{Cesaroni-Pagliari-flows}  and references therein.
 
 \subsection*{Anisotropic fractional curvature}
We recall that anisotropic perimeter related to a given convex body is a natural generalization of the classical perimeter and it is defined via a norm whose unit ball is equal to a given convex body. More precisely,
let $\mathcal{K}\subset \Rd$ be a convex compact set of non-empty interior (so-called convex body) and such that it is origin-symmetric. Let $\Vert \cdot \Vert_\mathcal{K}$ denote a norm on $\Rd$ given by
\begin{align*}
\Vert x\Vert_\mathcal{K} = \inf \{ \lambda >0: \lambda^{-1}x\in \mathcal{K}\},\quad x\in \Rd.
\end{align*}
Let $\mathcal{K}^*=\{y\in \Rd:\sup_{x\in \mathcal{K}} y\cdot x\leq 1\}$ be the polar body of $\mathcal{K}$. The anisotropic perimeter of a Borel set $E\subset \Rd$ with respect to $\mathcal{K}$ is defined as
\begin{align*}
\Per (E,\mathcal{K}) = \int_{\partial^* E}\Vert \mathsf{n}_{E}(x)\Vert_{\mathcal{K}^*}\, \ud x.
\end{align*}
Here $\mathsf{n}_E(x)$ denotes the measure theoretic outer unit normal vector of $E$ at $x\in \partial^* E$, where $\partial^*E$ is the reduced boundary of $E$, see \cite[Section 3.5]{Ambrosio_2000}.
For $0<\alpha <1$, the anisotropic $\alpha$-fractional perimeter of $E$ with respect to $\mathcal{K}$ is given by
\begin{align}\label{def-fract-anisotr-per}
\Per_{\alpha} (E,\mathcal{K}) = \int_E \int_{E^c}\frac{1}{\Vert x-y\Vert ^{d+\alpha }_\mathcal{K}}\,\ud x\, \ud y .
\end{align}
In \cite{Ludwig-Per} (see also \cite{Cygan-Grzywny-Per}) it was proved that for any bounded set of finite perimeter the following result holds
\begin{align}\label{anisotropic_alpha_per_coverg}
\lim_{\alpha \uparrow 1}(1-\alpha) \Per_\alpha (E,\mathcal{K}) = \Per (E,\mathcal{ZK}).
\end{align}
The set $\mathcal{ZK}$ is the so-called moment body of $\mathcal{K}$ which is defined as the unique convex body satisfying 
\begin{align}\label{moment_body}
\Vert y \Vert_{\mathcal{Z}^*\mathcal{K}} = \frac{d+1}{2}\int_\mathcal{K} |y\cdot x|\, \ud x,\quad  y\ \in \Rd,
\end{align}
where $\mathcal{Z}^*\mathcal{K}$ is the polar body of $\mathcal{ZK}$. 

The anisotropic fractional curvature of a Borel set $E\subset \Rd$ at $x\in \partial E$ related to the convex body $\mathcal{K}$ is defined as 
\begin{align}\label{anisotr-fract-curv}
H_{\alpha}^\mathcal{K} (x,E) = \lim_{r\downarrow 0}\int_{B_r(x)^c} \frac{\widetilde{\chi}_E(y)}{\|x-y\|_\mathcal{K}^{d+\alpha}}\, \ud y,
\end{align}
where 
\begin{align*}
\widetilde{\chi}_{E}(x) = \chi_{E^c} (x) - \chi_E (x),
\end{align*}
and $\chi_E (x)$ denotes the characteristic function of  $E\subset \R^d$.
Nonlocal curvatures of this type have been studied in the context of fractional mean curvature flows, see e.g.\ \cite{Imbert}, \cite{Chambolle-Interfaces}, \cite{Cesaroni-Stability}. One can show that the mean curvature in \eqref{anisotr-fract-curv} is the first variation of the perimeter defined at \eqref{def-fract-anisotr-per}, see \cite{Chambolle-Archiv}.
In the present paper we investigate the limit behaviour of \eqref{anisotr-fract-curv} as $\alpha \uparrow 1$. The question concerning the limit of the anisotropic fractional mean curvature was  posed in \cite[Remark 4.12]{Cesaroni-Stability} and the authors believed that uniform convergence holds and that the limit object should be the first variation of the anisotropic perimeter $\Per (E,\mathcal{ZK})$. We show that this conjecture is indeed true, that is 
we first prove that for any sequence $\alpha_n\uparrow 1$ and for sets $E_n\to E$ in $\Reg$ such that $x\in \partial E \cap \partial E_n$,
\begin{align}\label{anisotropic-result-intro}
\lim_{n\to \infty}(1-\alpha_n) H_{\alpha_n}^\mathcal{K} (x,E_n) =
\frac{1}{\kappa_{d-2}} \int_{e\in \Ss^{d-1}(x)\cap \, \nu(x)^{\perp}} \frac{K_{e}(x,E)}{\|e\|_\mathcal{K}^{d+1}}\, \calH^{d-2} (\ud e),
\end{align}
and, finally, we find that the first variation of the functional $\Per (E,\mathcal{ZK})$ is given by the right hand-side of \eqref{anisotropic-result-intro}, see Theorem \ref{anisotr-thm-first-var}.
Here, $\Reg$ is the class of subsets of $\Rd$ that can be obtained as the closure of an open set with compact $\calC^2$ boundary; and
by $E_n \to E$ in $\Reg$ we mean that there exists a sequence of diffeomorphisms $\{\Phi_n\}$ such that it converges to the identity in $\calC^2$ and  $E_n = \Phi_n(E)$. 
Further, $\kappa_{d-2} = \calH^{d-2}(\Ss^{d-2})$ and $\calH^{d-2} $ stands for the $(d-2)$-dimensional Hausdorff measure and $\Ss^{d-1}(x)$ is the unit sphere in $\R^{d}$ centred at $x$, and $\nu(x)$ is the normal unit vector for $\partial E$ at $x$, whereas $\nu (x)^\perp$ is the hyperplane orthogonal to $\nu (x)$. 

\subsection*{Classical mean and directional curvature}
To solve the posed problem of the convergence of anisotropic  fractional curvatures, we introduce a general version of nonlocal mean and directional curvatures, and study their asymptotics and related geometric flows. 
Before we discuss nonlocal versions of mean and directional curvatures, we briefly recall the definition of the classical directional curvature and the mean curvature.
Mean curvature of a surface in Euclidean space is a fundamental notion in classical differential geometry and it is an essential tool for the study of minimal surfaces, see \cite{Giusti}.  

Let us fix $S$ to be a $\calC^2$ surface in $\R^d$ and let $p\in S$ be a given point. The directional curvature $K_e(p,S)$ at $p$ along a unit vector $e$ lying in the tangent space to $S$ at $p$ is defined as the curvature at $p$ of the curve lying at the intersection of $S$ with the two-dimensional plane spanned by $e$ and a unit normal vector $\nu$ to $S$ at $p$. 

Suppose that $E\subset \R^d$ is a given set with $\calC^2$ boundary. For $p\in \partial E$  there is a small neighbourhood $B_r(p)$ around point $p$ such that the boundary $\partial E$ of $E$ can be represented as a graph of a $\calC^2$ function $f \colon B_r(p)\cap \R^{d-1}\to \R$ which has $p$ as a critical point, that is $f(p) =0$ and $\nabla f (p)=0$. The directional curvature of $E$ at $p\in \partial E$, denoted by $K_e(p,E)$, can be analytically determined via the following formula
\begin{align*}
K_e(p,E) =- D^2 f(p)e \cdot e,\quad e\in \Ss^{d-1}(p) \cap \nu(p)^{\perp}.
\end{align*}
Here $D^2 f(p)$ is the Hessian matrix of $f$ evaluated at $p$ and $x\cdot y$ stands for the scalar product of $x$ and $y$.
Since the matrix $D^2 f (p)$ is real and symmetric, it has $d-1$ real eigenvalues $\lambda_1,\ldots ,\lambda_{d-1}$ (they are called principal curvatures).
The mean curvature of $E$ at $p\in \partial E$ is denoted by $H(p,E)$ and it  is defined as the arithmetic mean of the principal curvatures (or equivalently as the normalized trace of the Hessian matrix, being equal, up to a constant, to the Laplacian of $f$), that is
\begin{align*}
H(p,E) = -\frac{\lambda_1 +\ldots +\lambda_{d-1}}{d-1} = - \frac{1}{d-1}\Delta f (p).
\end{align*}
If we take the average of the directional curvatures over the unit sphere we recover the mean curvature. More precisely, it holds
\begin{align}\label{mean-curv-through-direct}
H(p,E) = -\frac{1}{\kappa_{d-2}}\int_{\Ss^{d-1}(p) \cap \nu(p)^{\perp}}K_e(p,E)\, \calH^{d-2}(\ud e).
\end{align}
We remark that the minus appearing in the definition of the directional (and mean) curvature makes the curvature of any sphere positive.
It turns out that the mean curvature can be also computed through the following averaging procedure (see e.g.\ \cite[p.\ 172]{MR3824212})
\begin{align}\label{aver-proc}
-c_d \Delta f (p) = \lim_{r\to 0}\frac{1}{r^{d+1}}\int_{B_r(p)}\widetilde{\chi}_E(y)\ud y , 
\end{align}
where $c_d>0$ is a constant depending only on the dimension $d$.

%
%

\subsection*{Nonlocal curvature -- definition} We now present the definition of the nonlocal mean and directional curvature investigated in the present work. We consider a function $\phi \colon \Rd \to [0,\infty)$ such that 
\begin{equation}\label{eq:j-bound}
\phi(x)=\phi(-x) \quad \textnormal{and} \quad \phi(x) \leq j(x), \quad x\in\Rd,
\end{equation}
where $j\colon \Rd \to [0,\infty)$ is a given radial and radially non-increasing function
 satisfying\footnote{By $|x|$ we denote the Euclidean norm of $x\in \R^d$.}
\begin{equation}\label{eq:levy}
\int_{\Rd} (1 \wedge |x|^\beta) \, j(x) \, \ud x < \infty,\qquad \text{for some}\ \beta \in (0,1].
\end{equation}
For any set $E \subset \Rd$ with boundary of class $\calC^{1,\beta}$, the nonlocal mean curvature (corresponding to the kernel $\phi$) of $\partial E$ at $x \in \partial E$ is defined as
\begin{equation}\label{eq:mean}
H_{\phi} (x,E) =  \frac{1}{\kappa_{d-2}} \lim_{r\downarrow 0}  \int_{B_r(x)^c} \widetilde{\chi}_E(y) \phi(x-y) \, \ud y.
\end{equation}
We can also  define a version of a nonlocal directional curvature related to the kernel $\phi$. Let $\nu=\nu(x)$ and let $e$ be any unit vector in the tangent space
of $\partial E$ at $x$, and let $\piu(x,e)$ be the two-dimensional open half-plane through $x$ spanned by $e$ and $\nu$, that is
$$
\piu(x,e) = \{y \in \Rd : y = x + \rho e + h \nu, \, \rho > 0, h \in \R\}.
$$
The nonlocal directional curvature (related to the kernel $\phi$) in direction $e$ is defined as follows
\begin{equation}\label{eq:dir-sym}
K_{\phi,\, e} (x,E) = \frac{1}{2} \lim_{r\downarrow 0} \int_{B_r(x)^c\, \cap \, \pi (x,e)} \widetilde{\chi}_E(y)\, |Px-Py|^{d-2} \phi(x-y) \, \ud y,
\end{equation}
where $\pi(x,e) = \piu(x,e) \cup \piu(x,-e)$ 
and 
$P$ stands for the projection
onto the line through $x$ in the direction of $e$.


The integrals in \eqref{eq:mean} and \eqref{eq:dir-sym} 
are defined in the principal value sense and we show in Section \ref{sec:Prem} that because of cancellations these objects exist and are finite. We remark that the nonlocal curvature defined at \eqref{eq:mean} depends on the behaviour of the kernel $\phi$ only near the origin and thus we could consider the measure $\nu(\ud x) = \phi(x) \ud x + \overline{\nu}(\ud x)$ instead, where $\overline{\nu}$ is a measure in $\R^d$ such that $\operatorname{supp} (\overline{\nu}) \cap B_R = \emptyset$ for some ball $B_R$ centred at the origin and of radius $R>0$.

\subsection*{Approximation result} One of the main contributions of the present paper is the following approximation result for the nonlocal curvatures related to the kernel $\phi$ in the case when it takes a specific form. Let $\{j_{\eps}\}_{\eps>0}$ be a family of radial and radially non-increasing positive
functions that satisfy condition \eqref{eq:levy} with $\beta =1$ and let $\{g_{\eps}\}_{\eps>0}$ be a family of non-negative continuous functions defined on $\Ss^{d-1}$ such that 
\begin{align*}
g_{\eps}(-x) = g_{\eps} (x)\qquad \text{and}\qquad \lim_{\varepsilon \downarrow 0}g_{\eps}(e)=g(e)\quad
\text{uniformly in } e\in \Ss^{d-1},
\end{align*}
for a given continuous function $g \colon \Ss^{d-1} \to [0,\infty)$. We consider the following family of kernels 
$$
\phi_{\eps}(x) = g_{\eps}\left(\frac{x}{|x|}\right) j_{\eps}(x),  \qquad \eps >0.
$$
Set $C_{\eps} = \int_{\Rd} \left(1 \wedge |x|\right) \phi_{\eps}(x)\, \ud x$.
Under the assumption that the measures $ C_\eps^{-1} (1 \wedge |x|)\,  \phi_{\eps}(x) \, \ud x$ 
 concentrate at the origin (see condition \eqref{tight-zero} for details),
 we prove (see Theorem \ref{thm2}) that if
$E_n\to E$ in $\Reg$, then for any sequence $\varepsilon_n\downarrow 0$ and for every $x\in \partial E\cap \partial E_n$ it holds
\begin{align}\label{conv-nonlocal-mean-curv}
\lim_{n\to \infty} {C}_{\phi_{\eps_n}}^{-1}  K_{\phi_{\eps_n},e} (x,E_{n})
= C_g^{-1}g(e,0) K_{e}(x,E), \quad e\in \Ss^{d-1}(x) \cap \nu(x)^{\perp},
\end{align}
and 
\begin{align}\label{conv-nonlocal-direct-curv}
\lim_{n\to \infty} {C}_{\phi_{\eps_n}}^{-1}  H_{\phi_{\eps_n}} (x,E_n)
=  \frac{1}{C_g\kappa_{d-2}} \int_{e\in \Ss^{d-1}(x) \cap \nu(x)^{\perp}} g(e,0) K_{e}(x,E)  \calH^{d-2}(\ud e),
\end{align}
where $C_g=\int_{\Ss^{d-1}}g(\theta) \calH^{d-1}(\ud \theta)$.
Convergence of anisotropic $\alpha$-fractional curvatures displayed in \eqref{anisotropic-result-intro} is a direct consequence of \eqref{conv-nonlocal-direct-curv}, see the proof of Theorem \ref{thm:conv-anisotropic-curv} for details. 

\subsection*{Geometric flows and stability}
Another goal of the present work is to study geometric flows determined through the introduced model of nonlocal curvature. 
Recently the authors of \cite{Chambolle-Archiv} developed a unified approach for the study of nonlocal curvatures and corresponding geometric flows. In particular, they formulated conditions for a proper model of a nonlocal curvature such that the corresponding level-set Cauchy problem admits viscosity solutions. In this paper we undertake this approach and we show in Section \ref{subsec:axioms} that the nonlocal curvature $H_\phi$ satisfies the basic axioms of \cite{Chambolle-Archiv} and thus we can establish (see Section \ref{subsec:existence}) existence and uniqueness of continuous viscosity solutions to the following level-set equation
\begin{align}\label{Cauchy-eq-intro}
\begin{cases}
 \partial_t u(x,t) + |\nabla u(x,t)| H_\phi(x,  \{ y: \, u(y,t)\ge u(x,t)\}) \,=\,0 \\ 
 u(\cdot,0) = u_0(\cdot).
\end{cases}
\end{align}
Here $u_0 \colon \R^d\to\R$ is a given continuous function which is constant on the complement of a compact set.
Adopting another procedure from \cite{Chambolle-Archiv} we also show (see Section \ref{subsec:variation}) that the curvature $H_\phi$ is the first variation of the corresponding nonlocal perimeter which was recently studied in \cite{Cygan-Grzywny-Per}. 

Our convergence result \eqref{conv-nonlocal-mean-curv} allows us to adopt the approach developed in \cite{Cesaroni-Stability} which concerns the problem of stability of geometric flows generated by nonlocal curvatures. As already mentioned above, there exists a unique viscosity solution to \eqref{Cauchy-eq-intro} if we replace $H_\phi$ with $H_{\phi_{\eps_n}}$, so it makes sense to ask a question concerning convergence of such solutions. In Theorem \ref{thm:stability} we give a full answer to this question as we show that these viscosity solutions converge locally uniformly to the unique viscosity solution to \eqref{Cauchy-eq-intro} with curvature $H_\phi$ being replaced with the classical mean curvature $H$.

\subsection*{Nonlocal fractional and non-singular curvatures}
We briefly discuss another versions of nonlocal curvatures that are known in the literature and we comment on the links between these definitions and our approach. 
A prominent example of a nonlocal  mean curvature at $x\in E$ for a given set $E\subset\Rd$ with boundary $\partial E$ of class $\calC^2$ was proposed in \cite{Abatan-Valdin-curv} (see also \cite{Caffarelli_Savin} and \cite{Imbert}), where the authors adopted the procedure from \eqref{aver-proc} in the fractional setting. This is the so-called fractional mean curvature and it is defined as
\begin{align}\label{fract-mean-curv}
H_{\alpha}(x,E) = \frac{1}{\kappa_{d-2}} \lim_{r\downarrow 0}\int_{B_r(x)^c} \frac{\widetilde{\chi}_E(y)}{|x-y|^{d+\alpha}} \, \ud y, \quad x \in \partial E,\ \alpha \in (0,1).
\end{align}
Moreover, a version of the directional fractional curvature was also introduced in \cite{Abatan-Valdin-curv} and the authors established the fractional counterpart of \eqref{mean-curv-through-direct}. 
The fractional directional curvature is defined as
\begin{equation}\label{direc-curv-fract}
K_{\alpha,\, e} (x,E) = \lim_{r\downarrow 0} \int_{B_r(x)^c\, \cap \piu (x,e)} \widetilde{\chi}_E(y)\, |Px-Py|^{d-2} |x-y|^{-d-\alpha} \, \ud y,
\end{equation}
where 
$P$ stands for the projection
onto the line through $x$ in the direction of $e$.

Another variant of nonlocal curvature defined through a non-singular and integrable kernel was introduced in \cite{Rossi_paper_1} (see also \cite{Rossi_book}). More precisely, if $J \colon \Rd \to [0,\infty)$ is a radial and integrable function then the corresponding $J$-mean curvature is defined as
\begin{align}\label{def-J-curv}
H_J(x,E) = \int_{\Rd} \widetilde{\chi}_E(y) J(x-y) \, \ud y.
\end{align} 
It is worth noting that the integral at \eqref{def-J-curv} is well-defined  for all $x\in \R^d$ and not only along the boundary of the set $E$.

Important incentive for our work was the fact that one can approximate the classical mean (and directional) curvature with the aid of nonlocal fractional curvatures under appropriate rescaling. 
We may assume, without loss of generality, that we work in the normal system of coordinates, so that $x=0$ and the tangent space of $\partial E$ at $0$ is the hyperplane $\{x_d=0\}$. The unit normal vector at the origin can be taken to be $\nu =(0,\ldots ,0,1)$. 
If $x=0$ then we abandon $x$ when writing the symbols of  curvatures of $\partial E$, that is we write $H(E)$ for the classical mean curvature of $\partial E$ at $0$, $H_\alpha (E)$ for the fractional curvature of $\partial E$ at $0$ etc. 
For $\alpha$-fractional curvatures  
it was found in \cite[Theorem 12]{Abatan-Valdin-curv} that for any $E\subset \Rd$ with $\partial E \in \calC^2$ it holds
\begin{equation}\label{conv-fract-direct-to-class}
\lim_{\alpha \uparrow 1} (1-\alpha) K_{\alpha,e}(E) =K_e(E), \qquad e \in \Ss^{d-2},
\end{equation}
and similarly,
\begin{equation}\label{conv-fract-mean-to-class}
\lim_{\alpha \uparrow 1} (1-\alpha) H_{\alpha}(E) = H(E).
\end{equation}
In the case when $\alpha \downarrow 0$, the asymptotic behaviour of the fractional mean curvature was found in \cite[(B.1)]{MR3824212} for bounded sets with $\calC^2$ boundary. More precisely, if $E\subset B_R$ for some ball $B_R$ of radius $R>0$, then
\begin{equation}\label{conv-fract-at-0}
\alpha \kappa_{d-2} H_{\alpha} (E) = \kappa_{d-1} + \alpha \left(\int_{B_R} \frac{\tilde{\chi}_E(y)}{|y|^{d+\alpha}} \ud y - \kappa_{d-1} \log R\right) + o(\alpha),\quad \alpha \downarrow 0.
\end{equation}
Moreover, uniform versions of \eqref{conv-fract-mean-to-class} and \eqref{conv-fract-at-0} have been recently established in \cite{Cesaroni-Stability}, that is the authors introduced a mode of uniform convergence of nonlocal curvatures and allowed the set $E$ to evolve along a given sequence of diffeomorphisms.   
To be more precise, 
in \cite[Theorem 4.10]{Cesaroni-Stability} it was proved that for any sequence $\alpha_n \uparrow 1$, and if $E_n\to E$ in $\Reg$, and $0\in \partial E\cap \partial E_n$, it holds
\begin{align*}
\lim_{n\to \infty} (1-\alpha_n)H_{\alpha_n}(E_n) = H(E).
\end{align*} 
In \cite[Theorem 4.7]{Cesaroni-Stability} convergence in \eqref{conv-fract-at-0} was extended to the same uniform setting. 

We remark that an approximation result which is analogous to \eqref{conv-fract-mean-to-class} in the case of curvatures governed by non-singular kernels was established in \cite{Rossi_paper_1}. It was proved that if $J$ is non-negative, radially symmetric, continuous almost everywhere and compactly supported, then for any $E\subset \Rd$ with $\partial E$ of class $\calC^2$
the following convergence holds
\begin{equation}\label{conv_j-curv-to-class}
\lim_{\eps \da 0} \frac{C_J}{\eps} H_{J_{\eps}}(E) =  (d-1)\, H(E),
\end{equation}
where $J_{\eps}(x) = \eps^{-d} J(x/\eps)$ and 
$C_J = 2\left(\int_{\Rd} J(x)|x_d| \ud x\right)^{-1}$.

One can show that \eqref{conv-fract-direct-to-class} is a special case of \eqref{conv-nonlocal-direct-curv} and similarly, \eqref{conv-fract-mean-to-class} and \eqref{conv_j-curv-to-class} can be derived from \eqref{conv-nonlocal-mean-curv}.
If the (normalized) measures $ (1 \wedge |x|)\,  \phi_{\eps}(x) \, \ud x$ concentrate at infinity (see condition \eqref{cond-Lambda-one}) we obtain in the same setting  (see Theorem \ref{thm3}) that
\begin{align*}
\lim_{n\to \infty}C_{\eps_n}^{-1}H_{\phi_{\varepsilon_n}} (E_n)=\kappa_{d-2}^{-1},
\end{align*}
which can be used to recover the approximation from \eqref{conv-fract-at-0} and the part of \cite[Theorem 4.7]{Cesaroni-Stability} without the correction term. 

\subsection*{Other related works}
Geometric flows constitute an active  area of research 
in the modern calculus of variations. Classical mean curvature flows are treated in-depth in \cite{Mantegazza} and for fractional setting we refer to \cite{Imbert}, \cite{Chambolle-Interfaces}, \cite{Valdinoci-Flattering}, \cite{Julin} and \cite{Cesaroni-Stability}, and for models related to dislocation dynamics see \cite{JEMS} and \cite{Barles}. 

Convergence of anisotropic nonlocal curvatures defined through a given interaction kernel was recently studied in \cite{Cesaroni-Pagliari-flows}. The authors developed another interesting approach based on the theory of De Giorgi's barriers for evolution equations and they established convergence of rescaled anisotropic nonlocal curvatures to the classical anisotropic mean curvature. They also managed to show convergence of the corresponding geometric flows. To compare our assumptions with those of \cite{Cesaroni-Pagliari-flows}, we remark that we do not assume any regularity of the kernel $\phi$ apart from being controlled by the isotropic and monotone profile $j$ which satisfies \eqref{eq:levy} with $\beta =1$. The authors in \cite{Cesaroni-Pagliari-flows} assume that their interaction kernel (and its gradient) satisfy a chain of integrability conditions to assure that the corresponding curvatures are well-defined for sets with $\calC^{1,1}$ boundary.

\section{Preliminary results}\label{sec:Prem}
We start with the observation that we can define the nonlocal directional curvature corresponding to the isotropic profile $j$ similarly as in \eqref{direc-curv-fract}, that is
\begin{align}\label{eq:dir}
\Ku_{j,e} (x,E)= \lim_{r\downarrow 0} \int_{B_r(x)^c\, \cap \, \piu (x,e)} \widetilde{\chi}_E(y)\,|Px-Py|^{d-2} j(x-y) \, \ud y.
\end{align}
We note that for general $\phi$ satisfying \eqref{eq:j-bound}
the quantity in \eqref{eq:dir} may be not well-defined.
We additionally observe that $\Ku_{j,-e}(x,E)$ is, in general, not equal to $\Ku_{j,e}(x,E)$ which differs  from the classical case as then $K_{-e}(x,E)=K_e(x,E)$. 
We also note that $K_{j,e}$ may be seen as a symmetrization of $\Ku_{j,e}$ in the following manner 
\begin{align*}
K_{j,e} (x,E)= \frac{\Ku_{j,e}(x,E)+\Ku_{j,-e}(x,E)}{2}.
\end{align*}
We consider the normal system of coordinates in which $x=0$ and the tangent space of $\partial E$ at $0$ is the horizontal hyperplane $\{(x_1,\ldots ,x_d)\in \R^d:\, x_d = 0\}$. The normal unit vector is chosen as 
\begin{equation}\label{eq:v}
v =(0, \ldots, 0, 1).
\end{equation}
We also notice that the plane $\pi(e):=\pi(0,e)$ is endowed with the induced two-dimensional Lebesgue measure and the integration over
$\pi(e)$ is governed by the formula
\begin{equation}\label{eq:int}
\int_{\pi(e)} F(y) \, \ud y = \int_{\R^2}F(\rho e + h v) \, \ud h\, \ud \rho.
\end{equation}
In the fixed system of coordinates the objects defined at \eqref{eq:mean}, \eqref{eq:dir-sym} and \eqref{eq:dir} read as follows
\begin{equation}\label{eq:mean0}
H_{\phi}(E) =\frac{1}{\kappa_{d-2}}\lim_{r\downarrow 0} \int_{B_r^c} \widetilde{\chi}_E(y) \phi(y) \, \ud y,
\end{equation}
\begin{equation}\label{eq:dir-sym0}
K_{\phi,\, e} (E) = \frac{1}{2} \lim_{r\downarrow 0} \int_{B_r^c\, \cap \, \pi (e)} \widetilde{\chi}_E(y) \phi(y) |y'|^{d-2} \, \ud y,
\end{equation}
and
\begin{equation}\label{eq:dir0}
\Ku_{j,\, e} (E) = \lim_{r\downarrow 0} \int_{B_r^c\, \cap \, \piu (e)} \widetilde{\chi}_E(y) j(y) |y'|^{d-2} \, \ud y,
\end{equation}
where we use notation $\Rd \ni x = (x', x_d) \in \R^{d-1} \times \R$.

\begin{proposition}\label{curv-well-defined}
For any $E\subset \R^d$ with boundary of class $\calC^{1,\beta}$ and such that $E\subset \overline{\operatorname{int}(E)}$ the limits in \eqref{eq:mean0}, \eqref{eq:dir-sym0} and \eqref{eq:dir0} exist and are finite. 
\end{proposition}
\begin{proof}
We only prove that limits in \eqref{eq:mean0} and \eqref{eq:dir-sym0} exist as the reasoning for \eqref{eq:dir0} is similar. We remark that our proof is similar to that of \cite[Lemma 7]{Abatan-Valdin-curv} but it includes numerous important innovations and adjustments.
Without loss of generality we can consider the case when $E$ is closed (we remark that Lebesgue measure of its boundary $\partial E$ is zero) and it can be represented as the subgraph of a given function $f\in \calC^{1,\beta}(\R^{d-1})$ such that $f(0)=0$ and $\nabla f(0)=0$, so that we can write 
\begin{align}\label{E-subgraph}
E =\{ (y',y_d)\in \R^{d-1}\times \R:\, y_d \leq f(y')\}.
\end{align}
Since $E$ has boundary of class $\calC^{1,\beta}$, we can find a small neighbourhood of the origin with two tangent internal and external paraboloidal domains. More precisely, there exists a ball $B_R$ and a constant $M>0$ such that
\begin{align}\label{eq:1}
E \cap B_R \subset \{(y', y_d) \in \Rd: y_d \leq M|y'|^{1+\beta}\}\ \mathrm{and}\
E^c \cap B_R \subset \{(y', y_d) \in \Rd: y_d \geq -M|y'|^{1+\beta}\}.
\end{align}
We set
$E_R = B_R \cap \{y_d \leq f(y')\}$ and 
$\widetilde{E}_R = B_R \cap \{y_d > f(y')\}$.
We remark that it suffices to restrict our attention to the integral over the ball $B_R$ as outside of the ball the integral is evidently finite. 
For any $0<r<R$ we have
\begin{align*}
\int_{B_R \cap B_{r}^c} (\chi_{E^c}(y) -\chi_E(y)) \phi(y) \ud y &= \int_{\widetilde{E}_R\setminus B_{r}} \phi(y) \ud y 
- \int_{E_R\setminus B_{r}} \phi(y) \ud y  \\
&= \int_{-\widetilde{E}_R \setminus B_{r}} \phi(y) \ud y - \int_{E_R \setminus B_{r}} \phi(y) \ud y\\
&= \int_{((-\widetilde{E}_R)\setminus {E}_R)\setminus B_{r}} \phi(y) \ud y - \int_{(E_R\setminus (-\widetilde{E}_R))\setminus B_{r}} \phi(y) \ud y\\
&\stackrel{r \da 0}{\longrightarrow} \int_{(-\widetilde{E}_R)\setminus {E}_R} \phi(y) \ud y - \int_{E_R\setminus (-\widetilde{E}_R)} \phi(y) \ud y,
\end{align*}
where in the second equality we used the fact that $\phi (y)=\phi (-y)$ and in the last step we applied the monotone convergence theorem. We observe that 
\begin{align}\label{eq:f-bound-B_R}
f(y')\leq M|y'|^{1+\beta},\quad y\in B_r
\end{align}
and whence\footnote{Throughout we slightly abuse notation by writing $j(r)$ in place of $j(x)$ whenever $|x|=r$; and similarly we write $j(x')$ instead of $j((x',0))$.}, by \eqref{eq:j-bound} and the fact that $j$ is radially non-increasing,
\begin{align*}
\int_{E_R \setminus (-\widetilde{E}_R)} \phi(y) \ud y 
&=\int_{B_R \cap \{-f(-y') \leq y_d < f(y')\}} \phi(y) \ud y
\leq \int_{B_R \cap \{ |y_d|\leq M |y'|^{1+\beta}\} } j(y) \ud y \\
&\leq \int_{|y'|<R} \int_{ |y_d|\leq M |y'|^{1+\beta}}  j(y) \ud y_d\, \ud y'
\leq 2 M \int_{|y'|<R} |y'|^{1+\beta} j(y') \ud y'  \\
&= 2 M \kappa_{d-2} \int_0^R \rho^{d+\beta-1} j(\rho) \ud \rho 
=2M \frac{\kappa_{d-2}}{\kappa_{d-1}} \int_{B_R} |x|^{\beta} j(x) \ud x<\infty.
\end{align*}
Similarly, we have
\begin{align*}
\int_{(-\widetilde{E}_R)\setminus E_R} \phi(y) \ud y &
\leq \int_{B_R \cap \{ |y_d|\leq M |y'|^{1+\beta}\} } j(y) \ud y 
\leq 2M \frac{\kappa_{d-2}}{\kappa_{d-1}} \int_{B_R} |x|^{\beta}j(x) \ud x <\infty.
\end{align*}
This yields the finiteness of the quantity in \eqref{eq:mean0}. 

We next deal with \eqref{eq:dir-sym0}.
For any $0<r<R$ we have
\begin{align*}
\int_{(B_R \cap \pi(e))\setminus B_{r}} &(\chi_{E^c}(y)-\chi_E(y)) \phi(y) |y'|^{d-2} \ud y \\
&= \int_{(\widetilde{E}_R \cap \pi(e)) \setminus B_{r}} \phi(y) |y'|^{d-2} \ud y
-\int_{(E_R \cap \pi(e)) \setminus B_{r}} \phi(y) |y'|^{d-2} \ud y \\
&= \int_{(-\widetilde{E}_R \cap \pi(e)) \setminus B_{r}} \phi(y) |y'|^{d-2} \ud y
-\int_{(E_R \cap \pi(e)) \setminus B_{r}} \phi(y) |y'|^{d-2}\ud y \\
&=\int_{\left(((-\widetilde{E}_R)\setminus {E}_R)\cap \pi(e)\right)\setminus B_{r}} \phi(y) |y'|^{d-2} \ud y
-\int_{\left((E_R\setminus (-\widetilde{E}_R)) \cap \pi(e)\right) \setminus B_{r}} \phi(y) |y'|^{d-2} \ud y \\
&\stackrel{r \da 0}{\longrightarrow} \int_{\left((-\widetilde{E}_R)\setminus E_R\right) \cap \pi(e)} \phi(y) |y'|^{d-2} \ud y
-\int_{\left(E_R\setminus (-\widetilde{E}_R) \right)\cap \pi(e)} \phi(y)|y'|^{d-2} \ud y,
\end{align*}
where in the second equality we used the fact that $\phi (y)=\phi (-y)$ and in the last step we applied the monotone convergence theorem.  Using \eqref{eq:f-bound-B_R}, \eqref{eq:j-bound} and the fact that $j$ is radially non-increasing we obtain
\begin{align*}
\int_{\left(E_R \setminus (-\widetilde{E}_R)\right) \cap \pi(e)} \phi(y) |y'|^{d-2} \ud y 
&=\int_{B_R \cap \{-f(-y') \leq y_d < f(y')\} \cap \pi(e)} \phi(y) |y'|^{d-2} \ud y\\
&\leq \int_{B_R \cap \{ |y_d|\leq M |y'|^{1+\beta}\} \cap \pi(e)} j(y) |y'|^{d-2} \ud y \\
&\leq \int_{B_R \cap \{ |y_d|\leq M |y'|^{1+\beta}\} \cap \pi(e)} j(y') |y'|^{d-2} \ud y\\
&\leq \int_{-R}^{R} \int_{-M\rho^{1+\beta}}^{M\rho^{1+\beta}} j(\rho e) |\rho|^{d-2} \ud h \ud \rho \\
&\leq 4M \int_0^R j(\rho e) \rho^{d+\beta-1} \ud \rho
= \frac{4M}{\kappa_{d-1}} \int_{B_R} j(x) |x|^{\beta} \ud x <\infty.
\end{align*}
Similarly, we have
\begin{align*}
\int_{\left((-\widetilde{E}_R)\setminus E_R\right) \cap \pi(e)} \phi(y) |y'|^{d-2} \ud y &
\leq \int_{B_R \cap \{ |y_d|\leq M |y'|^{1+\beta}\} \cap \pi(e)} j(y) |y'|^{d-2} \ud y \\
&\leq \frac{4M}{\kappa_{d-1}} \int_{B_R} |x|^{\beta}j(x) \ud x <\infty
\end{align*}
and the proof is finished.
\end{proof}

The next result reveals that in the present context the nonlocal mean curvature can be recovered by integrating the nonlocal directional curvature over the unit sphere, analogously as in the classical case displayed in \eqref{mean-curv-through-direct}.

\begin{theorem}\label{thm:formula}
For any $E\subset \R^d$ with boundary of class $\calC^{1,\beta}$ and such that $E\subset \overline{\operatorname{int}(E)}$  it holds
$$
H_{\phi} (E) = \frac{1}{\kappa_{d-2}} \int_{\Ss^{d-2}} K_{\phi,\, e} (E)\,  \calH^{d-2}(\ud e).
$$
\end{theorem}

\begin{proof}
The proof is similar to that of \cite[Theorem 8]{Abatan-Valdin-curv} and it is based on formula \eqref{eq:int} where we apply it to the function $y \mapsto |y'|^{d-2}G(y)$ for $G \in L^1(\Rd)$. We work again in the normal coordinates chosen at the beginning of the present section  and we abuse  notation by identifying  $e = (e_1, \ldots, e_{d-1}, 0) \in \Rd$
with $(e_1, \ldots, e_{d-1}) \in \R^{d-1}$, which together with our choice \eqref{eq:v} allows us to write
$
\pi(e) \ni y = (\rho e, h)
$.
With this notation formula \eqref{eq:int} reads as follows
\begin{align*}
\int_{\pi(e)} |y'|^{d-2} G(y) \, \ud y = \int_{\R} \int_{\R} |\rho|^{d-2} G(\rho e, h) \, \ud h \, \ud \rho\quad \mathrm{a.s.\ for\ } e \in \Ss^{d-2}.
\end{align*}
We integrate the identity above over $\Ss^{d-2}$ and by applying the polar coordinates formula in $\R^{d-1}$ we arrive at
\begin{align*}
\int_{\Ss^{d-2}}\int_{\pi(e)} |y'|^{d-2} G(y) \, \ud y \calH^{d-2}(\ud e) 
&= \int_{\R} \int_{\Ss^{d-2}} \int_{\R} |\rho|^{d-2} G(\rho e, h)  \ud \rho \calH^{d-2}(\ud e) \ud h \\
&= 2 \int_{\R} \int_{\Ss^{d-2}} \int_0^{\infty} \rho^{d-2} G(\rho e, h)  \ud \rho\calH^{d-2}(\ud e) \ud h \\
&= 2 \int_{\R}\int_{\R^{d-1}} G(x', h) \, \ud x' \ud h = 2\int_{\Rd} G(x) \, \ud x.
\end{align*}
We apply the formula above to the function $G(y)= \chi_{B_r^c}(y) \widetilde{\chi}_E(y)\phi(y)$ and this yields
\begin{align*}
\int_{\Ss^{d-2}} K_{\phi,\, e} (E) \,  \calH^{d-2}(\ud e) 
&=\frac{1}{2}\int_{\Ss^{d-2}} \lim_{r\da 0} \, \int_{\pi(e)} \chi_{B_r^c}(y)\widetilde{\chi}_E(y) \phi(y) |y'|^{d-2} \, \ud y \calH^{d-2}(\ud e) \\
&=\frac{1}{2} \lim_{r \da 0} \int_{\Ss^{d-2}} \int_{\pi(e)} \chi_{B_r^c}(y) \widetilde{\chi}_E(y) \phi(y) |y'|^{d-2} \, \ud y \calH^{d-2}(\ud e)\\
&= \lim_{r \da 0}  \int_{\Rd}  \chi_{B_r^c}(y) \widetilde{\chi}_E(x) \phi(x) \, \ud x 
= \kappa_{d-2} H_{\phi} (E),
\end{align*}
where we used Lebesgue's dominated convergence theorem
to interchange the order of the limit and the integral. This can be justified similarly as in  the proof of Proposition \ref{curv-well-defined}, that is one can show that for any $0<r<R$ we have
\begin{align*}
&\left| \int_{ \pi(e)\setminus B_r} (\chi_{E^c}(y)-\chi_E(y)) \phi(y) |y'|^{d-2} \ud y \right| 
\leq \frac{8M}{\kappa_{d-1}} \int_{\R^d} (1\wedge |x|^\beta )j(x) \ud x <\infty
\end{align*}
and the proof is completed.
\end{proof}

Proceeding similarly as in the proof of Theorem \ref{thm:formula} we can show the following corollary.
\begin{corollary}
For any $E\subset \R^d$ with boundary of class $\calC^{1,\beta}$ and such that $E\subset \overline{\operatorname{int}(E)}$ it holds
$$
H_{j} (E) = \frac{1}{\kappa_{d-2}} \int_{\Ss^{d-2}} \Ku_{j\, e} (E)\,  \calH^{d-2}(\ud e).
$$
\end{corollary}

We next show  that if the set $E$ is a subgraph of a function as in \eqref{E-subgraph} then the nolocal directional curvature can be computed through the defining function, cf.\ \cite[Eq.\ (11)]{Abatan-Valdin-curv}.

\begin{proposition}\label{prop:K-formula}
If  a set $E$ is given by \eqref{E-subgraph} 
then
\begin{align*}
K_{\phi,e} (E)  
&= \frac{1}{2}\int_{\R} \int_{|f(\rho e)|}^{|f(-\rho e)|} \phi(\rho e, h) |\rho|^{d-2} \ud h\, \ud \rho - \frac{1}{2} \int_{\R}\int_{-f(\rho e)}^{f(\rho e)} \phi(\rho e, h) |\rho|^{d-2} \ud h\, \ud \rho .
\end{align*}
In particular, we have
\begin{equation}\label{eq:k_sg}
\Ku_{j,e} (E) = -2 \int_0^{\infty} \int_0^{f(\rho e)} j(\sqrt{\rho^2+h^2}) \rho^{d-2} \, \ud h\,  \ud \rho.
\end{equation}
\end{proposition}
\begin{proof}
We have
\begin{align*}
2K_{\phi,e} (E) 
&= \lim_{r\da 0} \int_{\R} \int_{\R} \chi_{B_r^c}(\rho e, h) \widetilde{\chi}_E(\rho e,h) \phi(\rho e, h) |\rho|^{d-2} \,  \ud h\ud \rho 
\end{align*}
We observe that $ \widetilde{\chi}_E(\rho e, h) = -1$ if $ h< -|f(\rho e)|$ and $\widetilde{\chi}_E(\rho e, h)=1$ if $h>|f(\rho e)|$.
Therefore,
\begin{align*}
\begin{split}
2K_{\phi,e} (E)  &= \lim_{r\da 0} \Bigg[\int_{\R} \left(\int_{|f(\rho e)|}^{\infty} - \int_{-\infty}^{-|f(\rho e)|} \right)  \chi_{B_r^c}(\rho e, h)\phi(\rho e, h) |\rho|^{d-2} \ud h \ud \rho \\
&\quad\quad\quad + \int_{\R} \int_{-|f(\rho e)|}^{|f(\rho e)|}  \chi_{B_r^c}(\rho e, h) \widetilde{\chi}_E(\rho e, h) \phi(\rho e, h) |\rho|^{d-2} \ud h \ud \rho\Bigg].
\end{split}
\end{align*}
Further,
\begin{align}
\int_{-|f(\rho e)|}^{|f(\rho e)|}  \chi_{B_r^c}(\rho e, h) \widetilde{\chi}_E(\rho e, h) &\phi(\rho e, h) |\rho|^{d-2} \ud h \nonumber\\
\begin{split}
&=
\begin{cases}
\displaystyle{-\int_{-f(\rho e)}^{f(\rho e)}  \chi_{B_r^c}(\rho e, h)\phi(\rho e, h) |\rho|^{d-2} \ud h}, & \textnormal{if}\,\, f(\rho e)\geq 0,\\
\displaystyle{\int_{f(\rho e)}^{-f(\rho e)}  \chi_{B_r^c}(\rho e, h) \phi(\rho e, h) |\rho|^{d-2} \ud h}, & \textnormal{if}\,\, f(\rho e)<0,
\end{cases}
\end{split}\label{eq:111}\\
&= \displaystyle{-\int_{-f(\rho e)}^{f(\rho e)}  \chi_{B_r^c}(\rho e, h) \phi(\rho e, h) |\rho|^{d-2} \ud h}.\label{eq:1112}
\end{align}
Since $\phi$ is symmetric, we have
\begin{align*}
\int_{\R} \int_{-\infty}^{-|f(\rho e)|} \chi_{B_r^c}(\rho e, h) \phi(\rho e, h) |\rho|^{d-2} \ud h \ud \rho &= \int_{\R} \int_{|f(-\rho e)|}^{\infty}\chi_{B_r^c}(\rho e, h) \phi(\rho e, h)   |\rho|^{d-2} \ud h \ud \rho .
\end{align*}
By the proof of Proposition \ref{curv-well-defined}, the limit of the integral of \eqref{eq:1112} exists,
hence
\begin{align*}
2K_{\phi,e} (E)  &=  \int_{\R} \int_{|f(\rho e)|}^{|f(-\rho e)|}
 \phi(\rho e, h) |\rho|^{d-2} \ud h \ud \rho - \int_{\R}\int_{-f(\rho e)}^{f(\rho e)} \phi(\rho e, h) |\rho|^{d-2} \ud h \ud \rho.
\end{align*}
For the case when $g = 1$ 
we apply \eqref{eq:int} to the function $y\mapsto |y'|^{d-2} j(y) \widetilde{\chi}_E(y)$ and we obtain
\begin{align*}
\int_{\pi(e)} |y'|^{d-2} j(y) \widetilde{\chi}_E(y) \, \ud y = \int_0^{\infty} \int_{\R} j(\sqrt{\rho^2+h^2}) \rho^{d-2} \widetilde{\chi}_E(\rho e, h) \, \ud h\,  \ud \rho. 
\end{align*}
We next observe that $ \widetilde{\chi}_E(\rho e, h) = -1$ if $ h< -|f(\rho e)|$ and $\widetilde{\chi}_E(\rho e, h)=1$ if $h>|f(\rho e)|$. We conclude that for any fixed $e\in\Ss^{d-2}$ the function
$
h \mapsto \rho^{d-2} j(\sqrt{\rho^2+h^2}) \widetilde{\chi}_E(\rho e, h)
$
is odd for $h \in \R \setminus [-|f(\rho e)|, |f(\rho e)|]$. Hence
\begin{align*}
\int_{\R \setminus [-|f(\rho e)|, |f(\rho e)|]} \rho^{d-2} j(\sqrt{\rho^2+h^2})  \widetilde{\chi}_E(\rho e, h) \, \ud h = 0.
\end{align*}
By \eqref{eq:111} and symmetry of $j$, 
\begin{align*}
\int_{-|f(\rho e)|}^{|f(\rho e)|} \!\!\! \rho^{d-2} j(\sqrt{\rho^2+h^2})  \widetilde{\chi}_E(\rho e, h) \, \ud h  &= 
-\int_{-f(\rho e)}^{f(\rho e)} j(\rho e, h) |\rho|^{d-2} \ud h \\
&= -2\int_0^{f(\rho e)}\!\!\!\! j(\sqrt{\rho^2+h^2}) \rho^{d-2} \ud h.
\end{align*} 
This implies
\begin{align*}
\int_{\pi(e)} |y'|^{d-2} j(y)  \widetilde{\chi}_E(y) \ud y = -2\int_0^{\infty} \int_0^{f(\rho e)} \rho^{d-2} j(\sqrt{\rho^2+h^2})\, \ud h\, \ud \rho,
\end{align*}
and \eqref{eq:k_sg} follows.
\end{proof}

\section{Mean curvature flows}

In this section we study general properties of the nonlocal curvature related to a given kernel $\phi$ satisfying \eqref{eq:j-bound}. We first show that it fulfils the set of axioms formulated in paper \cite{Chambolle-Archiv} and thus it can be seen as a proper model of nonlocal curvature. We further show that it is the first variation of the corresponding  nonlocal perimeter, and finally, we establish existence and uniqueness of level-set viscosity solutions to the related parabolic Cauchy problem. 

\subsection{Nonlocal curvature axioms}\label{subsec:axioms}
Let $\Reg$ be the class of subsets of $\Rd$ that can be obtained as the closure of an open set with compact $\calC^2$ boundary. 
By $E_n \to E$ in $\Reg$ we mean that there exists a sequence of diffeomorphisms $\{\Phi_n\}$ such that it converges to identity in $\calC^2$ and  $E_n = \Phi_n(E)$.  
In this section we always assume that the kernel $j$ satisfies condition \eqref{eq:levy} with $\beta =1$.

\begin{proposition}\label{axioms}
The following properties of the curvature $H_\phi$ are satisfied.\\
(M) Monotonicity: Let $E,F \in \Reg$ be such that $E\subset F$ and let $x\in\partial E \cap \partial F$. It holds
\[
H_\phi(x,F)\leq H_\phi(x,E);
\]
(T) Translation invariance: Let $E\in\Reg$, $x\in \partial E$ and $z\in\Rd$. It holds 
\[H_\phi(x,E) = H_\phi(x+z, E+z);
\]
(C') Uniform continuity: For any $R>0$ there exists a modulus of continuity $\omega_R$ such that 
for all $E\in\Reg$ satisfying both interior
and exterior ball condition of radius $R$ at $x\in\partial E$,
 and for all diffeomorphisms $\Phi \colon \R^d\to \R^d$ of class {$\calC^2$},
with $\Phi(y)=y$ for $|y-x|\geq 1$, it holds
$$
|H(x,E) - H(\Phi(x),\Phi(E))|\leq \omega_{R}(\|\Phi - \mathrm{Id}\|_{\calC^2}).
$$
In particular, condition (C') implies the following weaker continuity condition:\\
(C) Continuity: If $E_n \to E$ in $\Reg$ and $\partial E_n \ni x_n \to x \in \partial E$ then $H_\phi(x_n, E_n) \to H_\phi(x,E)$.
\end{proposition}

Before we present a proof of Proposition \ref{axioms} we need to establish the following four auxiliary lemmas. 

\begin{lemma}\label{lem43}
Let $\delta>0$. Then, there exists an increasing continuous function $\omega_{\delta}\colon [0,\infty) \to \R$ with $\omega_{\delta}(0)=0$ such that for every $\eta\geq 0$ it holds
\begin{equation}\label{int-modulus-cont}
\int_{B_{\eta}\setminus (B_{\delta}(\delta e_d)\cup B_{\delta}(-\delta e_d))} j(y) \ud y \leq \omega_{\delta}(\eta).
\end{equation}
\end{lemma}
\begin{proof}
We split the integral in \eqref{int-modulus-cont} as follows
\begin{align*}
\int_{B_{\eta}\setminus (B_{\delta}(\delta e_d)\cup B_{\delta}(-\delta e_d))} j(y) \ud y
&\leq \left(\int_{B_{\eta}\setminus B_{\delta/2}} + \int_{A} \right) j(y) \ud y
= {\rm I}_1 + {\rm I}_2,
\end{align*}
where 
$$
A = B_{\eta} \cap \{|y'| < \delta/2 \} \cap \{|y_d| \leq \delta -\sqrt{\delta^2-|y'|^2}\}.
$$
In view of \eqref{eq:levy} (with $\beta =1$) the integral ${\rm I_1}$ is finite.
To handle the integral ${\rm I_2}$, we observe that
$$
A \subset  B_{\eta} \cap \{|y'| < \delta/2\}\cap \left\{|y_d| \leq \frac{1}{\sqrt{3}}(1+\delta^{-1}) (|y'| \wedge |y'|^2)\right\}.
$$
This implies
\begin{align*}
{\rm I_{2}} 
&\leq \int_{|y'|\leq \eta \wedge \frac{\delta}{2}}  \int_{\{|y_d| \leq \frac{1}{\sqrt{3}}\left(1+\frac{1}{\delta}\right) (|y'| \wedge |y'|^2)\}}  j(y') \ud y_d\ud y'\\
&\leq \frac{2}{\sqrt{3}}\left(1+\frac{1}{\delta}\right) \int_{|y'|\leq \eta \wedge \frac{\delta}{2}} (|y'| \wedge |y'|^2)  j(y') \ud y'
\leq c \left(1 +\frac{1}{\delta}\right) \int_{B_{ \eta \wedge \frac{\delta}{2}}} (1\wedge |y|)j(y) \ud y
\end{align*}
and the proof is finished.
\end{proof}

\begin{lemma}\label{lem44}
Let $E \in \Reg$ and let $\delta>0$ be such that $B_{\delta}(-\delta e_d) \subset E$ and $B_{\delta}(\delta e_d) \subset E^c$. Then, for every $\eta>0$,
$$
\lim_{r\da 0} \int_{B_{\eta}\setminus B_r} \widetilde{\chi}_E(y) \phi(y) \ud y = \int_{B_{\eta}\setminus (B_{\delta}(\delta e_d)\cup B_{\delta}(-\delta e_d))} \widetilde{\chi}_E(y) \phi(y) \ud y.
$$ 
\end{lemma}

\begin{proof}
We have
\begin{align*}
\int_{B_\eta \setminus B_r} \widetilde{\chi}_E(y) \phi(y) \ud y &=
\int_{(B_{\eta} \setminus B_r) \setminus (B_{\delta}(\delta e_d)\cup B_{\delta}(-\delta e_d))} \widetilde{\chi}_E(y) \phi(y) \ud y \\
&\quad + \int_{(B_{\eta}\setminus B_r) \cap B_{\delta}(\delta e_d)} \phi(y) \ud y - \int_{(B_{\eta}\setminus B_r) \cap B_{\delta}(-\delta e_d)} \phi(y) \ud y.
\end{align*}
The two last integrals cancel out due to the symmetry of $\phi$. The claim follows by Lemma \ref{lem43} combined with Lebesgue's dominated convergence theorem.
\end{proof}

\begin{remark}
The proof above reveals that the statement of Lemma \ref{lem44} remains valid if we replace $B_\eta$ with $\R^d$.
\end{remark}

\begin{lemma}\label{lm:Lip-estimate}
Let $E\subset \R^d$ be a Lipschitz domain with localisation radius $r_0$ and Lipschitz constant $\lambda$
 and let $(\partial E)_r = \{ x \in \Rd: \operatorname{dist}(x,\partial E) \leq r\}$ denote the tubular neighbourhood of $\partial E$ of radius $r>0$. Then, there exists $C=C(\lambda, r_0, d)$ such that for every $x\in\partial E$ and for all 
 $$
 0<r<R<\frac{r_0}{2} \left({d-1} + \left({\lambda} \sqrt{d-1} +\sqrt{1+\lambda^2} \right)^2\right)^{-1/2}
 $$
 it holds
$$
|(\partial E)_r \cap B(x,R)| \leq C r R^{d-1}.
$$
\end{lemma}
\begin{proof}
The proof relies on some techniques from \cite[Proof of Lemma 3.1]{Dyda}. Since $E$ is a Lipschitz domain, for each $x\in \partial E$ there exist an isometry $T_x$ of $\Rd$ and a Lipschitz function $\vp_x \colon \R^{d-1} \to \R$ with Lipschitz constant not greater than $\lambda$ such that
$$
T_x(E) \cap B(T_x(x), r_0) = \{y: y_d >\vp_x(y')\} \cap B(T_x(x), r_0).
$$
Let $x \in \partial E$. Without loss of generality we can assume that $T_x = {\rm Id}$ (otherwise, we can apply the considerations below to the Lipschitz domain $T_x(E)$ and the point $T_x(x)$, and then come back to $E$ and $x$ using $T_x^{-1}$). For $x\in\Rd$ we consider the ''vertical'' distance from $y$ to the graph of $\vp_x$, that is
$$
V_x(y) = |y_d - \vp_x(y')|. 
$$
It is easy to see that 
$$
\frac{V_x(y)}{\sqrt{1+\lambda^2}} \leq \delta_E(y) \leq V_x(y), \quad \textnormal{for} \,\, y \in E \cap B(x,r_0/2),
$$
where $\delta_A(y) = \operatorname{dist} (y, A^c) = \operatorname{inf} \{|z-y|: z\in A^c\}$ for $y \in A\subset \Rd$. Similarly,
$$
\frac{V_x(y)}{\sqrt{1+\lambda^2}} \leq \delta_{E^c}(y) \leq V_x(y), \quad \textnormal{for} \,\, y \in E^c \cap B(x,r_0/2).
$$
For a set $D\subset \R^{d-1}$ and $\rho>0$ we define
$$
Q_x(D, \rho) = \{y\in E: y' \in D, 0<V_x(y) \leq \rho\}.
$$
We consider the case when $D= K_{2R} = \{y\in\R^{d-1}: |y_i -x_i|\leq R \,\, \textnormal{for} \,\, i=1,2,\ldots, d-1\}$ and $\rho = r \sqrt{1+\lambda^2}$, $r<R$. 
We fix
$$
R < \frac{r_0}{2} \left({d-1} + \left({\lambda} \sqrt{d-1} +\sqrt{1+\lambda^2} \right)^2\right)^{-1/2}.
$$
It may be checked that for such $R$ we have $Q_x(K_{2R}, r\sqrt{1+\lambda^2}) \subset E\cap B(x,r_0/2)$. 
Hence
\begin{align*}
((\partial E)_r \cap B(x,R) \cap E) &\subset
\{ y \in E: y' \in K_{2R}, \delta_E(y) < r\} \\
&\subset \{y\in E: y' \in K_{2R}, V_x(y) < r \sqrt{1+\lambda^2}\} = Q_x(K_{2R}, r \sqrt{1+\lambda^2}).
\end{align*}
This implies
\begin{align*}
|(\partial E)_r \cap B(x,R) \cap E| \leq |Q_x(K_{2R}, r \sqrt{1+\lambda^2})| &= \int_{K_{2R}} \int_{\vp_x(y')}^{\vp_x(y') + r\sqrt{1+\lambda^2}}\, \ud y_d \, \ud y'\\ &= 2^{d-1} \sqrt{1+\lambda^2} r R^{d-1}.
\end{align*}
Similar calculations reveal that
$$
|(\partial E)_r \cap B(x,R) \cap E^c| \leq 2^{d-1} \sqrt{1+\lambda^2} r R^{d-1}.
$$
We conclude that the desired estimate is valid with $C= 2^d \sqrt{1+\lambda^2}$.
\end{proof}

\begin{lemma}\label{lm:Lip-estimate2}
Let $E\subset \R^d$ be a Lipschitz domain with localisation radius $r_0$ and Lipschitz constant $\lambda$
 and let $(\partial E)_r = \{ x \in \Rd: \operatorname{dist}(x,\partial E) \leq r\}$ denote the tubular neighbourhood of $\partial E$ of radius $r>0$. For any $K>1$ there exist $c = c(d,\lambda)>0$ and $C=C(\lambda, r_0, d, K)>0$ such that for $r< c r_0$ and any $z\in\Rd$,
$$
|(\partial E)_r \cap B(z,K)| \leq C r.
$$
\end{lemma}

\begin{proof}
We fix a constant 
$$
c= \frac{1}{2}\left({d-1} + \left({\lambda} \sqrt{d-1} +\sqrt{1+\lambda^2} \right)^2\right)^{-1/2}.
$$
The family $\{B(x,c {r_0})\}_{x\in\partial E}$ 
is an open cover of the set $(\partial E)_r \cap B(z,K)$. We shall construct its finite open subcover as follows. 
We  consider the set $\bigtimes_{i=1}^{d} [z_i-K,z_i+K]$ which contains $B(z,K)$. We cover this set by cubes such that the lengths of their diagonals are equal to $c r_0 -r >0$. The number of such cubes is at most 
$
4\left(\left\lfloor \frac{K}{c r_0 - r}\right\rfloor+ 1\right)^2.
$ 
If any of these cubes contains points belonging to $\partial E$, then we pick one such point and add a ball of radius $cr_0$ around it to the considered subcover.
We note that such ball contains the whole cube and the tubular neighbourhood of its boundary of radius $r$. 
Hence, the number of balls in the finite subcover is at most $4\left(\left\lfloor \frac{K}{cr_0-r}\right\rfloor+ 1\right)^2$. 
By Lemma \ref{lm:Lip-estimate}, there exists $C=C(\lambda, r_0, d)>0$ such that for every $x\in\partial E$ it holds
$$
|(\partial E)_r \cap B(x,c {r_0})| \leq C r r_0^{d-1}.
$$
It follows that
$$
|(\partial E)_r \cap B(z, K)| \leq 4C r r_0^{d-1} \left(\left\lfloor \frac{K}{cr_0-r}\right\rfloor+ 1\right)^2
$$
and this finishes the proof.
\end{proof}

\begin{proof}[Proof of Proposition \ref{axioms}]
Properties (M) and (T) are straightforward and thus their proofs are omitted. 
We focus on property (C'). We use ideas from the proof of \cite[Proposition 4.5]{Cesaroni-Stability} but it requires numerous changes and adjustments to the present setting.
We fix $\rr>0$ and let $E\in\Reg$ and $x\in\partial E$ to be such that $E$ enjoys both interior
and exterior ball condition of radius $\rr$ at $x$. We denote by $B_\rr^{\inn}$ and $B_\rr^{\ou}$ the interior and exterior tangent balls  to $\partial E$ at $x$, respectively.
Let $\Phi \colon \R^d\to \R^d$ be a diffeomorphism of class {$\calC^2$}
with $\Phi(y)=y$ for $|y-x|\ge 1$.
We set $\widehat E:=\Phi(E)$ and notice that for $\|\Phi - \mathrm{Id}\|_{\calC^2}$ small enough $\widehat E$ satisfies interior and exterior ball condition at $\widehat x:=\Phi (x)$ with radius $\frac \rr 2$. We denote the corresponding tangent balls by $\widehat B_{\frac{\rr}{2}}^{\inn}$ and $\widehat  B_{\frac \rr 2}^{\ou}$.
In light of Lemma \ref{lem43} and Lemma \ref{lem44} with $\delta=\frac{\rr}{2}$, for every $\eta>0$ it holds 
\begin{equation}\label{menouno}
\left|\lim_{r\to 0^+} 
\int_{B_{\eta}(\widehat x) \setminus B_r(\widehat x)}
\widetilde{\chi}_{\widehat E}(y) \phi(\widehat x - y) \ud y\right|
+\left|\lim_{r\to 0^+}
\int_{B_{\eta}(x) \setminus B_r(x)} 
\widetilde{\chi}_{E}(y) \phi(x-y)\ud y\right| 
\leq4 \omega_{\frac{\rr}{2}}(\eta).
\end{equation}
Let $\eps>0$ and let $\eta_\eps$ be such that $4\omega_{\frac \rr 2}(\eta_\eps)\leq \frac\eps 4$ and $\omega_{\frac \rr 2}(2\eta_\eps)\leq \frac\eps 4$.
By \eqref{menouno} we have
\begin{align}
|H_{\phi}(x,E) &- H_{\phi}(\Phi(x),\Phi(E))|\nonumber
\\
&\leq
\left|\int_{B_{\eta_\eps}^c(x)} \widetilde{\chi}_E(y)\phi(x-y) \ud y 
- \int_{B_{\eta_\eps}^c(\widehat x)} \widetilde{\chi}_{\widehat E}(y)\phi(\widehat x-y) \ud y \right|
+\frac{\eps}{4} \nonumber \\
&\leq \left|\int_{B_{\eta_\eps}^c(x)}
\widetilde{\chi}_E(y) \phi(x-y) \ud y
- \int_{\Phi(B_{\eta_\eps}^c(x))}
\widetilde{\chi}_{\widehat E} (y) \phi(\widehat x - y) \ud y\right| \label{int1}\\
&\quad 
+\left|\int_{\Phi(B^c_{\eta_\eps}(x))}
\widetilde{\chi}_{\widehat E}(y) \phi(\widehat x -y) \ud y
- \int_{B^c_{\eta_\eps}(\widehat x)}
\widetilde{\chi}_{\widehat E}(y) \phi(\widehat x - y) \ud y\right| +\frac \eps 4, \label{int2}
\end{align}
where in the last step we simply used the triangle inequality.

Since $\Phi=\mathrm{Id}$ outside of $B_1(x)$, by a change of variables we obtain that the expression in \eqref{int1} can be estimated as follows
\begin{align*}
\left|\int_{B^c_{\eta_\eps}(x)}
\widetilde{\chi}_E(y) \phi(x-y)
\ud y - \int_{\Phi(B^c_{\eta_\eps}(x))}
\widetilde{\chi}_{\widehat E} (y) \phi(\widehat x - y) 
\ud y\right|
\leq {\rm I} + {\rm J},
\end{align*}
where 
\begin{align*}
{\rm I} = \left| \int_{B^c_1(x)} \widetilde{\chi}_{E} (y) \phi(x - y) \ud y
-  \int_{B^c_1(x)} \widetilde{\chi}_{E} (y) \phi(\widehat x - y) \ud y \right|
\end{align*}
and
\begin{align*}
{\rm J}= \left|\int_{B_1(x)\setminus B_{\eta_\eps}(x)}
\widetilde{\chi}_E(y) \phi(x-y)
\ud y - \int_{\Phi(B_1(x)\setminus B_{\eta_\eps}(x))}
\widetilde{\chi}_{\widehat E} (y) \phi(\widehat x - y) 
\ud y\right|.
\end{align*}
Translation by $\tau = x-\widehat{x}$ in the second integral yields
\begin{align*}
{\rm J} &=
\left|\int_{A}
\widetilde{\chi}_E(y) \phi(x-y)
\ud y - \int_{\widehat{A}_{\tau}} \widetilde{\chi}_{\widehat{E}_{\tau}} (y) \phi(x-y)
\ud y  \right|,
\end{align*}
where we use notation
$$
A = B_1(x)\setminus B_{\eta_{\eps}}(x), \qquad \widehat{A}_{\tau} =\widehat{A} +x -\widehat{x}, \qquad \widehat{E}_{\tau} = \widehat{E}+x-\widehat{x}.
$$
We set $\rho=\|\Phi - {\rm Id}\|_{\infty}$. We estimate $J$ as follows
\begin{align}
{\rm J} &= \left|\int_{\Rd} \left( \chi_{A\cap E^c} (y) - \chi_{A\cap E}(y) - \chi_{\widehat{A}_{\tau}\cap \widehat{E}_{\tau}^c}(y) + \chi_{\tauhA \cap \tauhE}(y) \right) \phi(x-y) \ud y\right|\nonumber\\
&\leq 
\int_{\Rd} \left| \chi_{A\cap E^c} (y) - \chi_{\tauhA \cap \tauhE^c}(y) \right| j(x-y) \ud y  
+
 \int_{\Rd} \left| \chi_{A\cap E}(y) - \chi_{\tauhA \cap \tauhE}(y) \right| j(x-y) \ud y
\nonumber \\
&\leq \left(\big|\big(A\cap E^c\big)\triangle \big(\tauhA \cap \tauhE^c\big)\big| 
+ \big|\big(A\cap E\big)\triangle \big(\tauhA \cap \tauhE\big)\big|\right) j(\eta_{\eps} - 2\rho). \label{measure-of-sets}
\end{align}
The estimate from \eqref{measure-of-sets} is justified by the fact that $j$ is radially non-increasing and for $y \in \tauhA$ and $z=\Phi^{-1}(y-x+\widehat{x})$ we have 
$$
|y-x| =|\Phi(z) -\widehat{x}| \geq |z-x| -|x-\widehat{x}| - |\Phi(z)-z| \geq\ \eta_{\eps} - 2\rho.
$$
We next estimate Lebesgue measure of sets in \eqref{measure-of-sets}. 
We begin with the set
\begin{align*}
\big(A\cap E\big) \setminus \big(\tauhA \cap \tauhE \big) 
&= \Big(\big(A\cap E\big) \setminus \tauhA \Big) \cup \Big(\big(A\cap E\big) \setminus \tauhE \Big) \\
&\subset \big(A\setminus \tauhA\big) \cup \Big(\big(E \setminus \tauhE\big) \cap B_1(x)\Big).
\end{align*}
We take $z\in E$ such that
$\operatorname{dist}(z,E^c)> 4 \rho$. For $y \in E^c$ it holds
$$
|\Phi(z) - \Phi(y)| \geq |z-y| - |\Phi(z)-z| - |y-\Phi(y)| > 2\rho,
$$ 
which implies $B(\Phi(z), 2\rho) \subset \widehat{E}$. Further, 
$$
|z - (\Phi(z) +x-\widehat{x})| \leq |z-\Phi(z)| + |x-\widehat{x}| \leq 2 \rho,
$$ 
yields
$
z\in B(\Phi(z) +z-\widehat{x}, 2\rho) \subset \tauhE.
$
We conclude that 
\begin{align*}
\big(E \setminus \tauhE\big) \cap B_1(x) \subset (\partial E)_{4\rho} \cap B_1(x).
\end{align*}
If, on the other hand, $z\in A$ is such that 
$\operatorname{dist}(z,A^c)> 4 \rho$, then for $y \in A^c$ we have
$$
|\Phi(z) - \Phi(y)| \geq |z-y| - |\Phi(z)-z| - |y-\Phi(y)| > 2\rho,
$$ 
which implies $B(\Phi(z),2\rho) \subset \widehat{A}$. Further, 
$$
|z-(\Phi(z)+x-\widehat{x})| \leq |z-\Phi(z)| + | x-\widehat{x}| \leq 2\rho,
$$
yields $z \in B(\Phi(z)+x-\widehat{x}, 2\rho) \subset \tauhA$. We conclude that
\begin{align*}
A\setminus  \tauhA \subset (\partial A)_{4\rho} \cap B_1(x).
\end{align*}
For the other half of the second symmetric difference in \eqref{measure-of-sets} we write
\begin{align*}
\big(\tauhA \cap \tauhE\big) \setminus\big(A\cap E\big)  &= 
\Big(\big(\tauhA\cap \tauhE\big) \setminus A\Big) \cup \Big(\big(\tauhA\cap \tauhE\big) \setminus E\Big) \\
&\subset \big(\tauhA\setminus A\big) \cup \Big( \big( \tauhE \setminus E\big) \cap \tauhA\Big) .
\end{align*}
If $z\in \tauhE$ is such that
$\operatorname{dist}(z, \tauhE^c)> 4 \rho$,
then, by setting $w= \Phi^{-1} (z-x+\widehat{x})$, we obtain for $y \in E^c$  
$$
|w - y| \geq |\Phi(w)-\Phi(y)| - |w- \Phi(w)| - |\Phi(y)-y |> 2\rho,
$$ 
which implies $B(w, 2\rho) \subset E$. Further, 
$$
|z -w| = |\Phi(w)+x-\widehat{x}-w| \leq  |\Phi(w) - w| + |x-\widehat{x}| \leq 2 \rho,
$$
yields $z\in B(w, 2\rho) \subset E$. 
We conclude that 
\begin{align*}
\tauhE \setminus E \subset \big(\partial (\tauhE)\big)_{4\rho} \cap \tauhA.
\end{align*} 
If, on the other hand, $z\in \tauhA$ and
$\operatorname{dist}(z,\tauhA^c)> 4 \rho$, 
then, by setting $w= \Phi^{-1} (z-x+\widehat{x})$, we have for $y \in A^c$
$$
|w-y| \geq |\Phi(w) - \Phi(y)| -|\Phi(y)-y| - |w-\Phi(w)| > 2\rho,
$$ 
which implies $B(w, \rho) \subset A$. Further, 
$$
|z-w| = |\Phi(w)+x-\widehat{x}-w| \leq |\Phi(w)-w| + | x-\widehat{x}| \leq 2\rho,
$$
yields $z \in B(w, 2\rho) \subset A$ and we obtain 
\begin{align*}
\tauhA\setminus A\subset \big( \partial ( \tauhA ) \big)_{4\rho}\cap \tauhA
\end{align*}
We arrive at
\begin{align*}
\big(A\cap E\big)\triangle \big(\tauhA \cap \tauhE\big) 
& \subset \big((\partial E)_{4\rho} \cap B_1(x)\big) 
\cup \big((\partial A)_{4\rho} \cap B_1(x)\big) \\
&\qquad \,
\cup \Big(\big(\partial (\tauhE)\big)_{4\rho}\cap \tauhA\Big)
\cup \Big(\big(\partial( \tauhA)\big)_{4\rho} \cap \tauhA\Big).
\end{align*}
Proceeding similarly for the first symmetric difference in \eqref{measure-of-sets}, we find that
\begin{align*}
\big(A\cap E^c\big)\triangle \big(\tauhA \cap \tauhE^c\big)
& \subset  \big(A\setminus \tauhA\big) \cup \big(E^c \setminus \tauhE^c\big) \cup
\big(\tauhA \setminus A\big) \cup \big(\tauhE^c \setminus E^c\big)\\
&\subset  \big((\partial A)_{4\rho}\cap B_1(x)\big) \cup \big((\partial E)_{4\rho} \cap B_1(x)\big)\\
&\qquad \cup \Big(\big(\partial (\tauhE)\big)_{4\rho} \cap \tauhA\Big)
\cup\Big(\big(\partial (\tauhA)\big)_{4\rho} \cap \tauhA\Big).
\end{align*}
Observe that if $z \in \tauhA$, that is $z= \Phi(y) + x - \widehat{x}$ for some $y\in A$, then
$$
|z-x| = |\Phi(y) -\widehat{x}| \leq |\Phi(y)-y| + |y-x| + |x-\widehat{x}| \leq 1+2\rho,
$$
hence $z\in B(x, 1+2\rho)\subset B(x,2)$. 

By a detailed inspection of the proof of \cite[Lemma 2.2]{Aikawa} we deduce that as $E$ satisfies the interior and exterior ball
conditions with radius $\rr$, it is a $\calC^{1,1}$ domain with $r_0=\rr/2$ and $\lambda$ depending only on $\rr$.
By Lemma \ref{lm:Lip-estimate2}, there exists $c_1= c_1(\rr, d)>0$ such that
$$
|(\partial E)_{4\rho} \cap B(x,1)| \leq 4c_1 \rho.
$$
Similarly, since $\widehat{E}$ satisfies the interior and exterior ball
conditions with radius $\rr/2$, it is a $\calC^{1,1}$ domain with $r_0=\rr/4$ and $\lambda$ depending only on $r_0$. 
Hence, there exists $c_2=c_2(\rr,d)>0$ such that
$$
\big|\big(\partial (\tauhE)\big)_{4\rho} \cap B(x,2)\big| \leq 4c_2 \rho.
$$
By Lemma \ref{lm:Lip-estimate2}, there exist $c_3= c_3(\eta_{\eps}, d)>0$ and $c_4= c_4(\eta_{\eps}, d)>0$ such that 
$$
|(\partial A)_{4\rho} \cap B(x,1)| \leq 4c_3 \rho \qquad \textnormal{and} \qquad
\big|\big(\partial (\tauhA)\big)_{4\rho} \cap B(x,1)\big| \leq 4c_4 \rho.
$$
We finally conclude that there is a constant $c_5 = c_5(\eta_{\eps}, \rr, d)>0$ such that
$$
{\rm J} \leq c_5 \rho j(\eta_{\eps} - 2\rho).
$$

We proceed similarly with the integral ${\rm I}$. Let $K>1$ be such that
$$
{\rm I_2} = \left| \int_{B^c_K(x)} \widetilde{\chi}_{E} (y) \phi(x - y) \ud y
-  \int_{B^c_K(x)} \widetilde{\chi}_{E} (y) \phi(\widehat x - y) \ud y \right| < \frac{\eps}{4}.
$$
We clearly have ${\rm I_1} = {\rm I} - {\rm I}_2$ and by a substitution we obtain
\begin{align*}
{\rm I_1} &=
\left|\int_{F}
\widetilde{\chi}_E(y) \phi(x-y)
\ud y - \int_{F_{\tau}} \widetilde{\chi}_{E_{\tau}} (y) \phi(x-y)
\ud y  \right|,
\end{align*}
where
$$
F = B_K(x)\setminus B_1(x), \qquad \tauF = F +x -\widehat{x}, \qquad \tauE ={E}+x-\widehat{x}.
$$
Using analogous arguments as in the case of ${\rm J}$, we arrive at
\begin{align*}
{\rm I_1} &\leq \left(\big|\big(F\cap E^c\big) \triangle \big(\tauF \cap \tauE^c\big)\big| + \big|\big(F\cap E\big) \triangle \big(\tauF \cap \tauE\big)\big| \right) j(1-\rho) \\
&\leq 2j(1-\rho) \left(|(\partial E)_{\rho} \cap B_K(x)\big|+ 
\big|(\partial F)_{\rho} \cap B_K(x)\big|  +\big|(\partial (\tauE))_{\rho} \cap \tauF\big|
+\big|(\partial (\tauF))_{\rho} \cap \tauF\big| \right)\\
&\leq c_6 \rho j(1-\rho),
\end{align*}
for a constant $c_6=c_6(K, \rr, d)>0.$

To estimate quantity \eqref{int2}, we observe that
for $\|\Phi - \mathrm{Id}\|_{\calC^1}$ small enough,
$$
\Big(\Phi\big(B_1(x)\setminus B_{\eta_\eps}(x)\big)\Big)
\triangle\Big(B_1(\widehat x)\setminus B_{\eta_\eps}(\widehat x)\Big)
\subset B_{2\eta_\eps}(\widehat x)\setminus 
\big(\widehat B_{\frac \rr 2}^{\inn}\cup \widehat B_{\frac \rr 2}^{\ou}\big).
$$
Hence, by Lemma \ref{lem43} we have
\begin{equation*}
\begin{aligned}
&\left|\int_{\Phi(B_1(x)\setminus B_{\eta_\eps}(x)}\widetilde{\chi}_{\widehat{E}} (y) \phi(\widehat{x}-y) \ud y -
\int_{B_1(\widehat{x})\setminus B_{\eta_{\eps}}(\widehat{x})}\widetilde{\chi}_{\widehat{E}} (y) \phi(\widehat{x}-y) \ud y\right|\\
&\leq \int_{B_{2\eta_\eps}(\widehat x)\setminus (\widehat B_{\frac \rr 2}^{\inn}\cup \widehat B_{\frac \rr 2}^{\ou})}
\phi(\widehat{x}-y) \ud y
\leq \omega_{\frac{\rr}{2}}(2\eta_\eps)\leq\frac{\eps}{4}
\end{aligned}
\end{equation*}
and the proof is completed.
\end{proof}

\subsection{The first variation of nonlocal perimeters}\label{subsec:variation}
We recall (see e.g.\ \cite{Cygan-Grzywny-Per}) that to the kernel $\phi $ we can attach a corresponding nonlocal perimeter which for a given Borel set $E\subset\Rd$ is defined as
\begin{align}\label{Per-phi-def}
\Per_{\phi}(E) = \int_E\int_{E^c} \phi(x-y)\, \ud y\,  \ud x.
\end{align}
We note that by \cite[Lemma 2.1]{Cygan-Grzywny-Per}, 
for any $E\subset \Rd$ of finite Lebesgue measure, we have $\Per_{\phi} (E) = \Per_{\phi} (E^c)$.  
One can easily verify that $\Per_{\phi}$ satisfies the axioms of generalized perimeter given in \cite{Chambolle-Archiv}. In particular, $\Per_{\phi}(E) < \infty$ for any $E \in \Reg$.
It turns out that for sets of finite Lebesgue measure and of finite classical perimeter, their $\phi$-perimeter can be computed through the formula 
\begin{align}\label{fromula-per}
\Per_{\phi}(E)= \mathcal{F}_{\phi}(\chi_E),\qquad \mathcal{F}_{\phi}(u) = \frac{1}{2} \int_{\Rd}\int_{\Rd} |u(x+y)-u(y)| \phi(x)\, \ud x\, \ud y.
\end{align}
We aim to show that the mean curvature $H_{\phi}$ is the first variation of $\Per_{\phi}$ in the sense of the following definition, see \cite{Chambolle-Archiv}.
A curvature $H(x,E)$, defined for $E\subset \Reg$ and $x\in\partial E$, is the first variation of a given perimeter $\Per$ if for every $E\subset \Reg$, and any one-parameter family of diffeomorphisms $(\Phi_{\eps})$ of class $\calC^2$ both in $x$ and $\eps$ with $\Phi_0(x)=x$, it holds
\begin{align}\label{def-curv-per}
\frac{\ud}{\ud\eps} \Per(\Phi_{\eps} (E)) \Big|_{\eps=0} = \int_{\partial E} H(x,E) \psi (x) \cdot \nu (x) \, \ud\calH^{d-1}(x),\quad \psi(x) = \frac{\partial \Phi_{\eps}}{\partial \eps} \Big|_{\eps=0}.
\end{align}

\begin{theorem}
The (normalized) nonlocal mean curvature $ H_\phi (x,E)$ is the first variation (in the sense of the definition given at \eqref{def-curv-per}) of the nonlocal $\phi$-perimeter defined at \eqref{Per-phi-def}.
\end{theorem}
\begin{proof}
Since $H_{\phi}(x,E)$ satisfies condition (C) of  Lemma \ref{axioms}, in view of  \cite[Proposition 4.5]{Chambolle-Archiv} the result will follow if we
show that for any $\vp\in \calC_c^2(\Rd)$, and $t_1<t_2$ such that $D\vp \neq 0$ in the set $\{t_1 \leq \vp \leq t_2\}$, it holds
\begin{align*}
\Per_{\phi}(\{\vp\geq t_1\})
=\Per_{\phi}(\{\vp\geq t_2\}) 
+\kappa_{d-2}
\int_{\{t_1 <\vp<t_2\}}  H_{\phi}(x, \{\vp\leq \vp(x)\})\, \ud x.
\end{align*}
Since $D\vp \neq 0$ in the set $\{t_1 \leq \vp \leq t_2\}$, we have $0 \notin (t_1,t_2)$. We note that $\{\vp \geq s\} \subset \Reg$ for any $s>0$.
If $0<t_1<t_2$, then according to \eqref{fromula-per}, 
\begin{align*}
\Per_{\phi}(\{\vp\geq t_1\}) - 	\Per_{\phi}(\{\vp\geq t_2\})
&=
\frac{1}{2} \int_{\Rd} \int_{\Rd} |\chi_{\{\vp\geq t_1\}} (x) - \chi_{\{\vp \geq t_1\}}(y)| {\phi}(x-y)\, \ud y\, \ud x\\
&\quad -\frac{1}{2} \int_{\Rd} \int_{\Rd} |\chi_{\{\vp \geq t_2\}} (x) - \chi_{\{\vp \geq t_2\}}(y)| {\phi}(x-y)\, \ud y\, \ud x.
\end{align*}
Next, for any  $\eps>0$ we have
\begin{align*}
&\frac{1}{2} \int_{\Rd} \int_{\Rd} |\chi_{\{\vp\geq t_1\}} (x) - \chi_{\{\vp \geq t_1\}}(y)| {\phi}(x-y) \chi_{|x-y|>\eps}(y) \, \ud y\, \ud x\\
&\quad -\frac{1}{2} \int_{\Rd} \int_{\Rd} |\chi_{\{\vp\geq t_2\}} (x) - \chi_{\{\vp\geq t_2\}}(y)| {\phi}(x-y) \chi_{|x-y|>\eps}(y) \, \ud y\, \ud x \\
&=\frac{1}{2} \int_{\Rd} \int_{\Rd} \chi_{\{t_1<\vp<t_2\}}(x) (\chi_{\{\vp\geq t_1\}}(y) - \chi_{\{\vp\geq t_2\}}(y)) {\phi}(x-y) \chi_{|x-y|>\eps}(y) \, \ud y\, \ud x \\
&\quad+\frac{1}{2} \int_{\Rd} \int_{\Rd} \chi_{\{t_1<\vp<t_2\}}(y) (\chi_{\{\vp\geq t_1\}}(x) - \chi_{\{\vp\geq t_2\}}(x)) {\phi}(x-y) \chi_{|x-y|>\eps}(y) \, \ud y\, \ud x \\
&= \int_{\Rd} \int_{\Rd} \chi_{\{t_1<\vp<t_2\}}(x) (\chi_{\{\vp \geq t_1\}}(y) - \chi_{\{\vp \geq t_2\}}(y)) {\phi}(x-y) \chi_{|x-y|>\eps}(y)  \, \ud y\, \ud x  \\
&= \int_{\{t_1<\vp<t_2\}} \int_{\Rd} (\chi_{\{\vp\leq \vp(x)\}^c} (y) -\chi_{\{\vp\leq \vp(x)\}}(y)) {\phi}(x-y) \chi_{|x-y|>\eps}(y)  \, \ud y\, \ud x   \\
&\quad
+\int_{\{t_1<\vp<t_2\}} \int_{\Rd}  (\chi_{\{\vp(x)<\vp<t_2\}}(y) -\chi_{\{t_1<\vp<\vp(x)\}}(y)) {\phi}(x-y) \chi_{|x-y|>\eps}(y)  \, \ud y\, \ud x \\
&=\int_{\{t_1<\vp<t_2\}} \int_{\Rd} (\chi_{\{\vp\leq \vp(x)\}^c} (y) -\chi_{\{\vp\leq \vp(x)\}}(y)) {\phi}(x-y) \chi_{|x-y|>\eps}(y)  \, \ud y\, \ud x ,
\end{align*}
where the last equality follows as the change of the order of integration and symmetry of ${\phi}$ imply
\begin{align*}
\int_{\{t_1<\vp<t_2\}} \int_{\{\vp(x)<\vp<t_2\}} & {\phi}(x-y) \chi_{|x-y|>\eps}(y)  \, \ud y\, \ud x \\
&= \int_{\{t_1<\vp<t_2\}} \int_{\{t_1<\vp<\vp(y)\}} {\phi}(x-y) \chi_{|x-y|>\eps}(y)\,  \ud x \,  \ud y\\
&= \int_{\{t_1<\vp<t_2\}} \int_{\{t_1<\vp<\vp(x)\}} {\phi}(x-y) \chi_{|x-y|>\eps}(y)\,  \ud y \,  \ud x.
\end{align*}
Since we know that $\vp \in \calC_c^2(\Rd)$ and $D\vp \neq 0$ in $\{t_1 \leq \vp \leq t_2\}$, we can apply the implicit function theorem to show that sets $\{\vp \leq \vp(x)\} -x$ satisfy uniform ball condition at $0$. Then a combination of Lemma \ref{lem43}, the proof of Lemma \ref{lem43} and Lebesugue's dominated convergence theorem yields the desired result for $\eps \da 0$.

If $t_1<t_2<0$, then $\{\vp \leq t_1\}, \{\vp \leq t_2\} \subset \Reg$. Moreover, $\Per_{\phi}(\{\vp\geq t_1\})$ and $\Per_{\phi}(\{\vp\geq t_2\})$ are equal to $\Per_{\phi}(\{\vp < t_1\})$ and $\Per_{\phi}(\{\vp <t_2\})$, respectively, whence we can easily repeat the calculations above in this case.
\end{proof}

\subsection{Existence of viscosity solutions}\label{subsec:existence}
By Proposition \ref{axioms} we know that $H_\phi$ satisfies the basic axioms (monotonicity, translation invariance and continuity) of \cite{Chambolle-Archiv} which
allow us to classify $H_\phi$ as a proper version of nonlocal curvature. In what follows we aim to study
existence, uniqueness and stability of viscosity solutions to the following parabolic Cauchy problem
\begin{equation}
\label{levelsetf}
\begin{cases}
 \partial_t u(x,t) + |\nabla u(x,t)| H_\phi(x,  \{ y: \, u(y,t)\ge u(x,t)\}) \,=\,0 \\ 
 u(\cdot,0) = u_0(\cdot),
\end{cases}
\end{equation}
where $u_0 \colon \R^d\to\R$ is a given continuous function which is constant on the complement of a compact set.
We remark that if the superlevel sets of $u$ are not smooth then the meaning of \eqref{levelsetf} is unclear and for this reason it is necessary to use a definition of a solution based on appropriate smooth test functions such that curvatures of their level sets  are
well defined. Such appropriate definition was proposed in \cite{Chambolle-Archiv} and this approach was undertaken also  in \cite[Section 2.2; Definition 2.3]{Cesaroni-Stability}. 
In order to establish existence of viscosity solutions to \eqref{levelsetf} we need to consider the following quantities, for any $\rho>0$,
\begin{equation}\label{defbarc0}
\overline{c}(\rho) \ :=\ 
\max_{x\in\partial B_\rho} \max\{H_\phi(x,\overline B_\rho),-H_\phi(x,\R^d\setminus B_\rho)\},
\end{equation}
\begin{equation}\label{defbarc}
\underline{c}(\rho) \ :=\ 
\min_{x\in\partial B_\rho} \min\{H_\phi(x, \overline B_\rho),-H_\phi(x,\R^d\setminus B_\rho)\}.
\end{equation}
These quantities are well-defined in light of property (C), and condition (M) implies that they are monotone functions of parameter $\rho$.

The following symmetry condition is easy to check.
\begin{itemize}
\item[(S)] Symmetry: for all $E \in \Reg$ and for every $x\in\partial E$, it holds $H_\phi(x,E) = - H_\phi(x,\operatorname{int}(E)^c)$.
\end{itemize}
Since the curvature $H_\phi$ is positive for any closed ball, the following condition extracted from \cite{Cesaroni-Stability} is trivially satisfied.
\begin{itemize}
\item[(B)] Ball condition: there is $K>0$ such that 
\begin{align*}
\underline{c}(\rho) \geq -K\rho,\quad \rho \geq 1.
\end{align*}
\end{itemize}
According to \cite[Theorem 2.9]{Cesaroni-Stability} conditions (M), (T), (C') and (B) imply the following result. 
\begin{theorem}\label{Thm:viscosity}
Let $H_\phi$ be the nonlocal curvature defined at  \eqref{eq:mean} for a kernel $\phi$ satisfying \eqref{eq:j-bound}.
Let $u_0$ be a given uniformly continuous function such that it is constant on the complement of a compact set. Then there exists a unique viscosity solution $u\colon \R^d\times [0,\infty)\to \R$ (in the sense of \cite[Definition 2.3]{Cesaroni-Stability}) to the problem \eqref{levelsetf}.
\end{theorem}

\section{Asymptotics of nonlocal curvatures}
This section is the core of the present article. We study the  asymptotic behaviour of the nonlocal curvature and we start with the case when it approaches  the classical curvature. 

In what follows, we associate the following measure to the kernel $\phi$
\begin{align}\label{meas-lamb}
\lambda_{\phi}(\ud x) = C_{\phi}^{-1} (1\wedge |x|) \phi(x) \ud x
\end{align}
where 
$C_{\phi} = \int_{\Rd} (1\wedge |x|) \phi(x) \ud x.$

\subsection{Convergence towards the classical curvature}

\begin{theorem}\label{thm2}
Let $\{j_{\eps}\}_{\eps>0}$ be a family of radial and radially non-increasing positive
functions such that for every $\eps >0$
$$
\int_{\Rd} \left(1\wedge |x|\right) j_{\eps} (x) \, \ud x < \infty.
$$
Further, let $\{g_{\eps}\}_{\eps>0}$ be a family of non-negative 
continuous functions defined on $\Ss^{d-1}$ such that $g_{\eps}(-x) = g_{\eps} (x)$ and $(g_{\eps})_{\eps>0}$ converges uniformly to a function $g \colon \Ss^{d-1} \to [0,\infty)$. We set 
$$
\phi_{\eps}(x) = g_{\eps}\left(\frac{x}{|x|}\right) j_{\eps}(x),  \qquad \eps >0.
$$ 
Suppose that
\begin{align}\label{tight-zero}
\lim_{\eps\downarrow 0} \lambda_{\phi_{\eps}} (B_R^c) = 0, \qquad \textnormal{for any} \,\, R>0.
\end{align}
Let $\varepsilon_n \downarrow 0$ and let $E_n \in \Reg$ be such that $E_n \to E$ in $\Reg$ for some $E\in \Reg$. Then, for every $x\in\partial E \cap \partial E_n$,
\begin{equation}\label{conv-dir-nonlocal-to classical}
\lim_{n\to \infty} {C}_{\phi_{\eps_n}}^{-1}  K_{\phi_{\eps_n},e} (x,E_{n})
= C_g^{-1}g(e,0) K_{e}(x,E), \quad e\in \Ss^{d-1}(x) \cap \nu(x)^{\perp},
\end{equation}
and 
\begin{equation}\label{conv-mean-nonlocal-to classical}
\lim_{n\to \infty} {C}_{\phi_{\eps_n}}^{-1}  H_{\phi_{\eps_n}} (x,E_n)
=  \frac{1}{C_g\kappa_{d-2}} \int_{e\in \Ss^{d-1}(x) \cap \nu(x)^{\perp}} g(e,0) K_{e}(x,E)  \calH^{d-2}(\ud e),
\end{equation}
where $C_g=\int_{\Ss^{d-1}}g(\theta) \calH^{d-1}(\ud \theta)$ and $\nu(x)$ is the outer unit normal vector to $\partial E$ at $x$.
\end{theorem}

\begin{proof}
Since $E_{n} \to E$ in $\Reg$, there exist $\delta>0$ and functions $f_{n}, f \in \calC^2 (B_{\delta}'; [0,\delta))$ such that $f_{n} \to f$ in $\calC^2(B_{\delta}', \R)$ such that $f_{n}(0)=f(0)=0, \nabla f_{n}(0)=\nabla f(0)=0$ and
\begin{align}
\begin{split}
&\partial E_{n} \cap B_{\delta} = \{ (y', f_{n}(y')): y' \in B_{\delta}'\}, \quad\quad\quad\quad
\partial E\cap B_{\delta} = \{ (y', f(y')): y' \in B_{\delta}'\},\\
&E_{n} \cap B_{\delta} = \{ (y', y_d): y' \in B_{\delta}', y_d \leq f_{n} (y') \}, \quad
E\cap B_{\delta} = \{ (y', y_d): y' \in B_{\delta}', y_d \leq f(y')\}.\label{eps-balls}
\end{split}
\end{align}
Without loss of generality we use the normal system of coordinates chosen at the beginning of Section \ref{sec:Prem} in which $x=0$ and \eqref{eq:v} holds. We can also assume that locally at the origin the set $E$ is a subgraph of a function exactly as in \eqref{E-subgraph}.
Such assumption
can be easily dropped if we observe that
the contribution to $K_{\phi_{\eps_n}, e}(E_n)$ coming from 
far away from the origin is bounded uniformly when $\eps_n \downarrow 0$, and whence it does not contribute to the limit. 
We easily find  that 
\begin{align*}
C_{\phi_{\eps}} = \kappa_{d-1}^{-1}C_{j_{\eps}} C_{g_{\eps}},
\end{align*} 
where $C_{g_{\eps}} = \int_{\Ss^{d-1}} g_{\eps}(\theta) \calH^{d-1}(\ud \theta)$. 
We have for any $\eps >0$
\begin{align*}
2K_{\phi_{\eps},e} (E_{n}) &= \int_{\R} \int_{\R} \widetilde{\chi}_{E_{n}}(\rho e,h) \phi_{\eps}(\rho e, h) |\rho|^{d-2} \,  \ud h\, \ud \rho \\
&=\int_{\R} \int_{\R} \widetilde{\chi}_{E_{n}}(\rho e,h) j_{\eps}(\sqrt{\rho^2+ h^2}) g_{\eps}\left(\frac{\rho e, h}{\sqrt{\rho^2+h^2}}\right) |\rho|^{d-2} \,  \ud h\, \ud \rho.
\end{align*}
For any $\eta\in(0,1)$ we split the last integral as follows
\begin{align}\label{I+J}
2K_{\phi_{\eps},e} (E_{n}) = \left(\int_{|\rho |<\eta} \int_{\R} + \int_{|\rho|>\eta}\int_\R \right)\widetilde{\chi}_{E_{n}}(\rho e,h) \phi_{\eps}(\rho e, h) |\rho|^{d-2} \,  \ud h\, \ud \rho = I+J.
\end{align}
Since $g_{\eps}$ converge uniformly to $g\geq 0$, it holds $|g_{\eps}|\leq c$, for some constant $c>0$ and $\eps>0$ small enough. We thus obtain that for $\varepsilon >0$ small enough,
\begin{align}\label{est-J}
\left|J\right| &
\leq 2c \int_{|\rho|>\eta}\int_{\R} |\rho|^{d-2} j_{\eps}(\sqrt{\rho^2+h^2}) \, \ud h \ud \rho  \nonumber \\
&\leq  2c \iint_{\R^2\setminus B_\eta}  (\rho^2+h^2)^{(d-2)/2} j_{\eps}(\sqrt{\rho^2+h^2}) \, \ud h \ud \rho 
=4\pi c  \int_{\eta}^{\infty} r^{d-1} j_{\eps}(r) \, \ud r \nonumber\\
&\leq 4\pi c\eta^{-1}\int_\eta^\infty (1\wedge r)r^{d-1}j_\varepsilon (r)\ud r
=4\pi c\, C_{g_{\eps}}^{-1} \eta^{-1} {C}_{\phi_{\eps}} {\lambda}_{\phi_{\eps}}(B_{\eta}^c).
\end{align}
We next handle integral $I$ from \eqref{I+J}. We use the fact that for $\eta \in (0,1)$ sufficiently small, by \eqref{eps-balls},
\begin{equation}\label{eq:Deps}
D_e^2f_{n}(0) \frac{\rho^2}{2} - w_{n}(\rho) \leq f_{n}(\rho e ) \leq D^2_e f_{n}(0) \frac{\rho^2}{2} + w_{n} (\rho), \quad \rho \in (-\eta, \eta)\setminus\{0\},
\end{equation}
where 
$$
w_{n}(\rho) = \sup_{|\xi ' | \leq |\rho|} |D^2_e f_{n}(\xi ') - D^2_e f_n(0)|  \frac{\rho^2}{2}.
$$
Note that $w_{n}(-\rho)=w_{n}(\rho)$.
We also set
$$
\mathcal{W}_{n}(\eta) = \sup_{0<|\rho|\leq\eta} \frac{w_{n}(\rho)}{\rho^2} = \frac{1}{2} \sup_{|\xi'|\leq \eta} |D_e^2 f_{n}(\xi') -  D_e^2 f_{n}(0)|
$$
and we observe that 
$$ 
\mathcal{W}_{n}(\eta)  \to 0,\quad \textnormal{as} \quad \eta \to 0.
$$
In view of Proposition \ref{prop:K-formula}, we can split integral $I$ as follows
\begin{align*}
I= \int_{-\eta}^{\eta} \int_{|f_{n}(\rho e)|}^{|f_{n}(-\rho e)|}\phi_{\eps}(\rho e, h) |\rho|^{d-2} \ud h \ud \rho  - \int_{-\eta}^{\eta} \int_{-f_{n}(\rho e)}^{f_{n}(\rho e)}  \phi_{\eps}(\rho e, h) |\rho|^{d-2} \ud h \ud \rho
=I_1 + I_2.
\end{align*}
By \eqref{eq:Deps} and since $j_{\eps}$ is radial and radially non-increasing, we obtain
\begin{align}\label{est-I_1}
|I_1|
&\leq 
c \int_{-\eta}^{\eta}  \Big||f_{n}(-\rho e)| - |f_{n}(\rho e)|\Big| j_{\eps}(\rho) |\rho|^{d-2} \ud \rho \nonumber 
\leq 
2c \int_{-\eta}^{\eta} w_{n}(\rho) j_{\eps}(\rho) |\rho|^{d-2} \ud \rho\\
&\leq 
4c \mathcal{W}_{n}(\eta)  \int_{0}^{\eta} j_{\eps}(\rho) \rho^{d} \ud \rho
\leq 4c  C_{g_{\eps}}^{-1} \mathcal{W}_{n}(\eta)    {C}_{\phi_{\eps}}.
\end{align}
To estimate $I_2$
we set $D_{n}=\frac{1}{2} D_e^2 f_{n}(0)$ and $D=\frac{1}{2} D_e^2 f(0)$. By \eqref{eq:111}, we have
\begin{align}\label{est-I_2}
|I_2- 4D_ng_{\eps}(e)C_{g_{\eps}}^{-1}  C_{\phi_{\eps}}| 
&\leq \left|\int_{-\eta}^{\eta} \int_{-D_{n}\rho^2}^{D_{n}\rho^2} \phi_{\eps}(\rho e, h) |\rho|^{d-2}\, \ud h \ud \rho 
- 4D_{n}g_{\eps}(e) C_{g_{\eps}}^{-1}  C_{\phi_{\eps}}\right|
\nonumber \\
&\qquad
+\left|\int_{-\eta}^{\eta} \left(\int_{-f_{n}(\rho e)}^{-D_{n}\rho^2} 
+ \int_{D_{n}\rho^2}^{f_{n}(\rho e)}\right) \phi_{\eps}(\rho e, h) |\rho|^{d-2} \, \ud h \ud \rho \right| \nonumber\\
&\leq |I_3^{(1)}|+ |I_3^{(2)}|+  4|D_{n}| g_{\eps}(e) C_{g_{\eps}}^{-1} C_{\phi_{\eps}} \lambda_{\phi_{\eps}}(B_{\eta}^c) ,
\end{align}
where 
\begin{align*}
I_3^{(1)}
&=
\int_{-\eta}^{\eta} \int_{-D_{n}\rho^2}^{D_{n}\rho^2} \phi_{\eps}(\rho e, h) |\rho|^{d-2} \, \ud h \ud \rho 
- 4D_{n}g_{\eps}(e)C_{g_{\eps}}^{-1}  C_{\phi_{\eps}} \lambda_{\phi_{\eps}}(B_{\eta})\\
I_3^{(2)}
&=
\int_{-\eta}^{\eta} \left(\int_{-f_{n}(\rho e)}^{-D_{n}\rho^2} 
+ \int_{D_{n}\rho^2}^{f_{n}(\rho e)}\right) \phi_{\eps}(\rho e, h) |\rho|^{d-2} \, \ud h \ud \rho .
\end{align*}
By \eqref{eq:Deps},
\begin{align}\label{est-I_3^2}
I_3^{(2)} &\leq \int_{-\eta}^{\eta} \left|\left(\int_{-f_{n}(\rho e)}^{-D_{n}\rho^2} + \int_{D_{n}\rho^2}^{f_{n}(\rho e)}\right) \phi_{\eps}(\rho e, h) \, \ud h \right| |\rho|^{d-2} \ud \rho \nonumber \\
&\leq 2c \int_{-\eta}^{\eta} |f_{n}(\rho e) - D_{n}\rho^2| j_{\eps}(\rho) |\rho|^{d-2}  \, \ud \rho \nonumber \\
&\leq 2c \int_{-\eta}^{\eta} w_{n}(\rho) |\rho|^{d-2} j_{\eps}(\rho) \, \ud \rho
= 4c \int_0^{\eta} w_{n}(\rho) \rho^{d-2} j_{\eps}(\rho) \, \ud \rho \nonumber \\
&\leq 4c \mathcal{W}_{n}(\eta)  \int_{0}^{\eta} j_{\eps}(\rho) \rho^d \, \ud \rho 
\leq 4 c C_{g_{\eps}}^{-1} \mathcal{W}_{n}(\eta) C_{\phi_{\eps}}.
\end{align}
Further,
\begin{align*}
\frac{1}{2} I_3^{(1)}
&\leq
\left|\int_{0}^{\eta} \int_{-D_{n}\rho^2}^{D_{n}\rho^2} j_{\eps}(\sqrt{\rho^2+h^2}) g_{\eps} 
\left(\frac{(\rho e, h)}{\sqrt{\rho^2+h^2}} \right) \rho^{d-2} \,  \ud h\ud \rho 
- 2D_{n} g_{\eps}(e)\int_0^{\eta} j_{\eps}(\rho) \rho^{d} \ud \rho\right| \\
&\leq \left|\int_{0}^{\eta} \int_{-D_{n}\rho^2}^{D_{n}\rho^2} j_{\eps}(\sqrt{\rho^2+h^2}) 
\left( g_{\eps}\left(\frac{(\rho e, h)}{\sqrt{\rho^2+h^2}} \right) - g_{\eps}(e)\right) \rho^{d-2} \,  \ud h\ud \rho\right| \\
&\qquad+
|g_{\eps}(e)| \left| \int_0^{\eta} \left[ \rho^{d-2} \int_{-D_{n}\rho^2}^{D_{n}\rho^2} j_{\eps}(\sqrt{\rho^2+h^2}) \ud h
- 2D_{n} j_{\eps}(\rho) \rho^{d} \right]\ud \rho\right| \\
&=I_3^{(1,1)} + I_3^{(1,2)}.
\end{align*}
To estimate the first integral we observe that for $0<\rho <\eta$ and $|h| < D_n\rho^2$ it holds
\begin{align*}
 \left|g_{\eps}\left(\frac{\rho e, h}{\sqrt{\rho^2+h^2}} \right) - g_{\eps}(e)\right| 
&\leq \omega_{g_{\eps}} \Big(\sqrt{2-2/\sqrt{1+D_{n}^2\eta^2}}\Big) ,
\end{align*}
where $\omega_F$ stands for the modulus of continuity of a function $F$, i.e. 
$$
\omega_F(\delta) = \sup \{|F(x)-F(y)|: x,y \in \Ss^{d-1} \, \textnormal{and} \, |x-y|\leq \delta\}.
$$
For a function $F$ continuous on $\Ss^{d-1}$ we clearly have $\omega_{F}(\delta) \to 0$ as $\delta\to 0$. 
We obtain
$$
 I_3^{(1,1)} \leq \omega_{g_{\eps}}\Big(\sqrt{2-2/\sqrt{1+D_{n}^2\eta^2}}\Big)
\left|\int_{0}^{\eta} \int_{-D_{n}\rho^2}^{D_{n}\rho^2} j_{\eps}(\sqrt{\rho^2+h^2}) \rho^{d-2} \ud h \ud \rho\right|.
$$
Moreover,
\begin{align*}
\left|\int_{0}^{\eta} \int_{-D_{n}\rho^2}^{D_{n}\rho^2} j_{\eps}(\sqrt{\rho^2+h^2}) \rho^{d-2} \ud h \ud \rho\right| 
&\leq 2D_{n} \int_0^{\eta} j_{\eps}(\rho) \rho^d \ud \rho 
= 2D_{n} \kappa_{d-1}^{-1} C_{j_{\eps}} \lambda_{j_{\eps}} (B_{\eta})\\
&\leq 2D_{n} \kappa_{d-1}^{-1} C_{j_{\eps}}
= 2D_{n}  C_{g_{\eps}}^{-1} C_{\phi_{\eps}},
\end{align*}
and whence
\begin{align}\label{est-I_3^11}
 I_3^{(1,1)}\leq  2D_{n} C_{g_{\eps}}^{-1} \omega_{g_{\eps}}\Big(\sqrt{2-2/\sqrt{1+D_{n}^2\eta^2}}\Big) {C}_{\phi_{\eps}}.
\end{align}
To estimate $I_3^{(1,2)}$ we proceed as follows.
We first observe that the function $h \mapsto j_{\eps}(\sqrt{\rho^2+h^2})$ is even and whence
\begin{multline*}
\left|\rho^{d-2} \int_{-D_{n}\rho^2}^{D_{n}\rho^2} j_{\eps}(\sqrt{\rho^2+h^2}) \, \ud h - 2D_{n}\rho^d j_{\eps}(\rho)\right| \\
= 2\left|\rho^{d-2} \int_0^{|D_{n}|\rho^2} j_{\eps}(\sqrt{\rho^2+h^2}) \, \ud h - |D_{n}|\rho^d j_{\eps}(\rho)\right|.
\end{multline*}
Since the same function as above is also non-increasing, we obtain
\begin{align*}
|D_{n}|\rho^d j_{\eps}(\rho\sqrt{1+D_{n}^2\rho^2}) \leq \rho^{d-2} \int_0^{|D_{n}|\rho^2} j_{\eps}(\sqrt{\rho^2+h^2}) \, \ud h \leq |D_{n}|\rho^d j_{\eps}(\rho).
\end{align*}
This implies
\begin{align*}
\left|\rho^{d-2} \int_0^{|D_{n}|\rho^2} j_{\eps}(\sqrt{\rho^2+h^2})\, \ud h - |D_{n}| \rho^d j_{\eps}(\rho)\right|
&= |D_{n}|\rho^d j_{\eps} (\rho) - \rho^{d-2} \int_0^{|D_{n}|\rho^2} j_{\eps} (\sqrt{\rho^2+h^2})\, \ud h\\
&\leq |D_{n}|\left(\rho^d j_{\eps} (\rho) -\rho^d j_{\eps}(\rho\sqrt{1+D_{n}^2\rho^2})\right).
\end{align*}
Further, by monotonicity of $j_{\eps}$, 
\begin{align*}
 \int_0^{\eta}  \rho^d j_{\eps}(\rho\sqrt{1+D_{n}^2\eta^2}) \, \ud \rho
&= \frac{1}{(\sqrt{1+D_{n}^2\eta^2})^d }\int_0^{\eta\sqrt{1+D_{n}^2\eta^2}} \rho^d j_{\eps}(\rho) \, \ud \rho\\
&\geq  \frac{1}{(\sqrt{1+D_{n}^2\eta^2})^d }\int_0^{\eta} \rho^d j_{\eps}(\rho) \, \ud \rho \geq  \left(1- \frac{d}{2} D_{n}^2\eta^2\right) \int_0^{\eta} \rho^d j_{\eps}(\rho) \,\ud \rho.
\end{align*}
Therefore, for any $\eta<1$,
\begin{align*}
I_3^{(1,2)}
&= 2\left|\int_0^{\eta} \left[ \rho^{d-2} \int_0^{D_{n}\rho^2} j_{\eps}(\sqrt{\rho^2+h^2}) \, \ud h - D_{n} \rho^{d}j_{\eps}(\rho) \right] \ud \rho\right|\\
&\leq 2|D_{n}| \left(\int_0^{\eta} \rho^d j_{\eps}(\rho)\, \ud\rho - \int_0^{\eta} \rho^d j_{\eps}(\rho\sqrt{1+D_{n}^2\rho^2})\,\ud \rho \right)\\
&\leq d |D_{n}|^3 \eta^2 \int_0^{\eta} \rho^d j_{\eps}(\rho) \, \ud \rho \leq  d |D_{n}|^3 \eta^2 C_{\phi_{\eps}},
\end{align*}
and we  conclude that
\begin{align}\label{est-I_3^12}
 I_3^{(1,2)} \leq |g_{\eps}(e)| d \kappa_{d-1}^{-1} |D_{n}|^3 \eta^2 C_{j_{\eps}}  
= C_{g_{\eps}}^{-1} |g_{\eps}(e)| d|D_{n}|^3 \eta^2 {C}_{\phi_{\eps}}.
\end{align}
Combining \eqref{est-J}, \eqref{est-I_1}, \eqref{est-I_2}, \eqref{est-I_3^2}, \eqref{est-I_3^11}  and \eqref{est-I_3^12} yields
\begin{align*}
\left|{C}_{\phi{\eps}}^{-1} K_{\phi_{\eps},e} (E_{n})
- 2D_{n}C_{g_\varepsilon}^{-1}g_{\eps}(e)\right| 
&\leq
4cC_{g_{\eps}}^{-1}\mathcal{W}_{n}(\eta) 
+ 2|D_{n}| g_{\eps}(e) C_{g_{\eps}}^{-1} \lambda_{\phi_{\eps}}(B_{\eta}^c)\\ 
&\quad + 2D_{n}C_{g_{\eps}}^{-1} \omega_{g_{\eps}}\Big(\sqrt{2-2/\sqrt{1+D_{n}^2\eta^2}}\Big) \\ 
&\qquad+C_{g_{\eps}}^{-1}  |g_{\eps}(e)| d|D_{n}|^3 \eta^2
+\pi c C_{g_{\eps}}^{-1}  \frac{1}{\eta} {\lambda}_{\phi_{\eps}}(B_{\eta}^c).
\end{align*}
We note that for $\delta>0$,
$$
\omega_{g_{\eps}}(\delta) \leq \omega_{g} (\delta) + 2 \sup\{|g_{\eps}(x) - g(x)| : x \in \Ss^{d-1}\}.
$$
Since $g_{\eps}$ converges uniformly to $g$, we conclude that for any $\varepsilon_n \downarrow 0$
\begin{multline*}
\limsup_{n \to \infty}  \left| C_{\phi_{\eps_n}}^{-1} K_{\phi_{\eps_n},e} (E_{n})
-2D_{n}C_{g_{\eps_n}}^{-1}g_{\eps_n}(e) \right|\\
\leq
\limsup_{n\to \infty}\left( 
4cC_{g_{\eps}}^{-1} \mathcal{W}_{n}(\eta) 
+ C_g^{-1} \kappa_{d-1} \omega_g\Big(\sqrt{2-2/\sqrt{1+D_{n}^2\eta^2}}\Big)
+C_g^{-1} \kappa_{d-1} |g(e)| d|D_{n}|^3 \eta^2\right).
\end{multline*}
Since $\eta>0$ can be taken arbitrarily small and the function $g$ is continuous, we arrive at
$$
\lim_{n\to \infty} \left|{C}_{\phi_{\eps_n}}^{-1} K_{\phi_{\eps_n},e} (E_{n})
-2D_{n}C_{g_{\eps_n}}^{-1}g_{\eps_n}(e)\right| =0.
$$
Since 
$
\lim_{\eps \da 0} C_{g_\varepsilon}^{-1}g_{\eps}(e)= C_g^{-1}g(e),
$
and
$\lim_{n\to \infty} D_{n} = D$,
we finally obtain
\begin{align}\label{K-conv-unif}
\lim_{n\to \infty} {C}_{\phi_{\eps_n}}^{-1} K_{\phi_{\eps_n},e} (E_{n}) = 2DC_g^{-1}g(e)
= C_g^{-1}g(e) K_e.
\end{align}
We  observe that the convergence in \eqref{K-conv-unif} is uniform 
over $\Ss^{d-2}$ and this allows us to deduce 
\begin{align*}
\lim_{n\to \infty} {C}_{\phi_{\eps_n}}^{-1} H_{\phi_{\eps_n}}(E_{n})
&=
 \lim_{n\to \infty} {C}_{\phi_{\eps_n}}^{-1} \frac{1}{\kappa_{d-2}} \int_{\Ss^{d-2}} K_{\phi_{\eps_n},e} (E_{n}) \calH^{d-2} (\ud e) \\
 &= \frac{1}{C_g \kappa_{d-2}} \int_{\Ss^{d-2}} g(e) K_e \calH^{d-2} (\ud e),
\end{align*}
and the proof is complete.
\end{proof}

In the corollary below we indicate the classical directional and mean curvatures in the limit in the case when $\phi=j$. In this case we can use definition \eqref{eq:dir} for the directional curvature. 

\begin{corollary}\label{cor-j-profiles}
Let $\{j_{\eps}\}_{\eps>0}$ be a family of radial and radially non-increasing functions such that
$$
\int_{\Rd} \left(1\wedge |x|\right) j_{\eps} (x) \ud x < \infty.
$$
Let
$$
\lambda_{\eps} (\ud x) = C_{\eps}^{-1} (1\wedge |x|) \, j_{\eps} (x) \ud x,
$$
where $C_{\eps} = \int_{\Rd} \left(1 \wedge |x|\right) j_{\eps} ( x) \ud x$.
Suppose that
\begin{equation}\label{tight_at_zero}
\lim_{\eps\da 0} \lambda_{\eps} (B_R^c) = 0, \qquad \textnormal{for any} \,\, R>0.
\end{equation}
Let $\varepsilon_n \downarrow 0$ and $E_{n} \in \Reg$ be such that $E_{n} \to E$ in $\Reg$ for some $E$ in $\Reg$. Then, for every $x\in\partial E \cap \partial E_{n}$,
\[
\lim_{n\to \infty} C_{\phi{\eps_n}}^{-1} \Ku_{j_{\eps_n},e} (x,E_n) = \frac{1}{\kappa_{d-1}}  K_{e}(x,E), \quad e\in \Ss^{d-1}(x) \cap \nu(x)^{\perp},
\]
and
\[
\lim_{n\to \infty} C_{{\eps_n}}^{-1}  H_{j_{\eps_n}} (x,E_n)= \frac{1}{\kappa_{d-1}} H(x,E).
\]
\end{corollary}

The following result
allows us to conclude the validity of the approximation of the classical curvatures by families of kernels that are perturbed by another admissible kernels, cf.\ Example \ref{ex10}.
\begin{proposition}\label{prop_nu}
Let $\{\phi_{\eps}\}_{\eps>0}$ satisfy the assumptions of Theorem \ref{thm2} and $\{\overline{\phi}_{\eps}\}_{\eps>0}$ are such that
$$
\phi_{\eps}(x)+\overline{\phi}_{\eps}(x) \geq 0, \quad x\in \Rd
$$
and
$$
\overline{\phi}_{\eps}(x) = \overline{\phi}_{\eps} (-x) \quad \textnormal{and} \quad |\overline{\phi}_{\eps}(x)| \leq \overline{j}_{\eps}(x), \quad x\in\Rd,
$$
for some family $\{\overline{j}_{\eps}\}_{\eps>0}$ of radial, radially non-increasing and non-negative functions satisfying
\begin{equation}
\int_{\Rd} (1 \wedge |x|) \, \overline{j}_{\eps}(x) \, \ud x < \infty.
\end{equation}
Suppose that
\begin{equation}\label{assumption}
\lim_{\eps \da 0} C_{\phi_{\eps}}^{-1} \int_{\Rd} (1\wedge |x|) \overline{j}_{\eps}(x) \ud x = 0.
\end{equation}
Let $\varepsilon_n \downarrow 0$ and let $E_n \in \Reg$ be such that $E_n \to E$ in $\Reg$ for some $E\in \Reg$. Then, for every $x\in\partial E \cap \partial E_n$,
$$
\lim_{n\to\infty}  C_{\phi_{\eps_{n}}}^{-1} K_{\phi_{\eps_n}+\overline{\phi}_{\eps_n},e}(x,E_n) = C_g^{-1}g(e) K_e(x,E)
$$
and
$$
\lim_{n \to \infty} C_{\phi_{\eps_{n}}}^{-1} H_{\phi_{\eps_n}+\overline{\phi}_{\eps_n}}(x,E_n) = \frac{1}{C_g \kappa_{d-2}} \int_{\Ss^{d-2}} g(e) K_e \calH^{d-2} (\ud e).
$$
\end{proposition}

\begin{proof}
Proceeding similarly as in the proof of Proposition \ref{curv-well-defined}, according to \eqref{eq:1} we can define $M_n= M(E_n)$ and $R_n=R(E_n)$ and we infer the following estimate
\begin{align*}
|K_{\phi_{\eps}+\overline{\phi}_{\eps},e}(E_n) - K_{\phi_{\eps},e}( E_n)| 
&\leq \frac{8M}{\kappa_{d-1}} \int_{B_R} |x|  \overline{j}_{\eps}(y) \ud y + \int_{B_R^c} \overline{j}_{\eps}(y) \ud y \\
&\leq c_n \int_{\Rd}(1\wedge |y|) \overline{j}_{\eps}(y) \ud y,
\end{align*}
where $c_n=2 \kappa_{d-1}^{-1} \left(4M_n(1\vee R_n) \vee \pi(1 \vee R_n^{-1})\right)$. One can show that constants $c_n$ are bounded from above by a positive constant $C$ which depends only on the set $E$ and $\max_{n\in \mathbb{N}} \|\Phi_n-\rm{Id}\|$. By \eqref{assumption} we conclude that for any $\varepsilon_n \downarrow 0$, 
$$
\frac{1}{C_{\phi_{\eps_n}}}\left|K_{\phi_{\eps_n}+\overline{\phi}_{\eps_n},e}(E_n) - K_{\phi_{\eps_n},e}(E_n)\right|  \leq \frac{C}{C_{\phi_{\eps_n}}} \int_{\Rd} (1\wedge |y|) \overline{j}_{\eps_n}(y) \ud y \stackrel{n\to \infty}{\longrightarrow} 0.
$$
By Theorem \ref{thm:formula},
$$
\frac{1}{C_{\phi_{\eps_n}}}\left|H_{\phi_{\eps_n}+\overline{\phi}_{\eps_n}}(E_n) - H_{\phi_{\eps_n}}(E_n)\right| 
\leq\frac{C}{C_{\phi_{\eps_n}}} \int_{\Rd} (1\wedge |y|) \overline{j}_{\eps_n}(y) \ud y\stackrel{n\to \infty}{\longrightarrow} 0.
$$
Theorem \ref{thm2} yields the result. 
\end{proof}

\subsection{Stability of mean curvature flows}
Let $E_n\in \Reg$ be such that $E_n\to E\in \Reg$. In view of Theorem \ref{thm2}  the curvatures $C_{\phi_{\varepsilon_n}}^{-1}H_{\phi_{\eps_n}}$ converge 
in the sense of \cite[Definition 3.1]{Cesaroni-Stability} to the nonlocal curvature given by the right-hand side of \eqref{conv-mean-nonlocal-to classical} which evidently satisfies condition (C'). 
In order to establish stability of the corresponding viscosity solutions to \eqref{levelsetf} we shall investigate the following condition (UB). 
Let $\underline{c}^n$ and $\overline{c}^n$
be defined as in  \eqref{defbarc0} and \eqref{defbarc} where $H_\phi$ is replaced with $H_{\phi_{\eps_n}}$. We set 
\begin{equation}\label{defucinfty}
\underline{c}^{\inf} (\rho):= \inf_{n\in\mathbb{N}}  \underline{c}^n (\rho), \qquad \overline{c}^{\sup} (\rho):= \sup_{n\in\mathbb{N}}  \overline{c}^n (\rho). 
\end{equation}  
\begin{itemize}
\item[(UB)] There exists $K\ge 0$ such that
$\underline{c}^{\inf} (\rho) \ge -K\rho$, for all $\rho>1$, and  
$\overline{c}^{\sup} (\rho)<+\infty$, for all $\rho>0$.
\end{itemize}

\begin{lemma}\label{lem:UB}
In the notation of Theorem \ref{thm2}, 
the curvatures $(C_{\phi_{\eps_n}}^{-1}H_{\phi_{\eps_n}})_{n\in\mathbb{N}}$ satisfy (UB).
\end{lemma}
\begin{proof}
We first remark that the $\phi$-curvature of any closed ball is always positive, so evidently $\underline{c}^{\inf} (\rho) \geq 0$.
Further, 
\begin{align*}
H_{\phi_{\eps_n}}(x,E) = H_{\phi_{\eps_n}}(0,E-x). 
\end{align*}
Hence, for any $\eta\in(0,1)$,
\begin{align*}
H_{\phi_{\eps_n}}(x,\overline{B}_\rho) 
= H_{\phi_{\eps_n}} (0, \overline{B}_\rho(-x))
&= \lim_{r\to 0^+} \int_{B_r^c} \widetilde{\chi}_{\overline{B}_\rho(-x)} (y) \phi_{\eps_n}(y) \ud y \\
&= \lim_{r\to 0^+} \left(\int_{B_{\eta}^c} + \int_{B_{\eta}\setminus B_{r}}\right) \widetilde{\chi}_{\overline{B}_\rho(-x)} (y) \phi_{\eps_n}(y) \ud y.
\end{align*}
We note that the interior and exterior tangent balls to $\overline{B}_\rho(-x)$ at $0$ are ${B}_\rho(-x)$ and ${B}_\rho(x)$, respectively. Hence, by Lemma \ref{lem44},
\begin{align*}
H_{\phi_{\eps_n}}(x,\overline{B}_\rho) 
&= \int_{B_{\eta}^c} \widetilde{\chi}_{\overline{B}_\rho(-x)} (y) \phi_{\eps_n}(y) \ud y 
+ \int_{B_{\eta}\setminus ({B}_\rho(-x) \cup {B}_\rho(x))} 
\widetilde{\chi}_{\overline{B}_\rho(-x)} (y) \phi_{\eps_n}(y) \ud y\\
&= \int_{B_{\eta}^c} \widetilde{\chi}_{\overline{B}_\rho(-x)} (y) \phi_{\eps_n}(y) \ud y 
+ \int_{B_{\eta}\setminus ({B}_\rho(-x) \cup {B}_\rho(x))} \phi_{\eps_n}(y) \ud y 
={\rm I}_1 + {\rm I}_2.
\end{align*}
We have
\begin{align*}
C_{\phi_{\eps_n}}^{-1} |{\rm I}_1| \leq C_{\phi_{\eps_n}}^{-1} \int_{B_{\eta}^c} \phi_{\eps_n}(y) \ud y \leq \eta^{-1} C_{\phi_{\eps_n}}^{-1} \int_{B_{\eta}^c} (1\wedge |y|) \phi_{\eps_n}(y) \ud y = \eta^{-1} \lambda_{\phi_{\eps_n}}(B_{\eta}^c), 
\end{align*}
which, by our assumptions, tends to $0$ as $n\to\infty$. Moreover, similarly as in the proof of Lemma \ref{lem43} we obtain 
\begin{align*}
C_{\phi_{\eps_n}}^{-1} 
{\rm I}_2 
&= C_{\phi_{\eps_n}}^{-1} 
\int_{B_{\eta}\setminus (B_{\rho}(-x) \cup B_{\rho}(x))} \phi_{\eps_n}(y) \ud y\\
&\leq C_{\phi_{\eps_n}}^{-1} 
\left( 
\int_{B_{\eta}\setminus B_{\rho/2}} \phi_{\eps_n}(y) \ud y 
+ c \rho \int_{B_{ \eta \wedge \frac{\rho}{2}}}(1\wedge |y|) \phi_{\eps_n}(y) \ud y 
\right) \\
&\leq C_{\phi_{\eps_n}}^{-1} 
\left( 
\frac{2}{\rho} \int_{B_{\eta}\setminus B_{\rho/2}} |y| \phi_{\eps_n}(y) \ud y
+ c\rho \int_{B_{ \eta \wedge \frac{\rho}{2}}} |y|\phi_{\eps_n}(y) \ud y \right) \\
&= 
\frac{2}{\rho}\lambda_{\phi_{\eps_n}}(B_{\eta}\setminus B_{\rho/2}) +  c \rho \lambda_{\phi_{\eps_n}}(B_{ \eta \wedge \frac{\rho}{2}}) 
\leq 
\frac{2}{\rho}
+ c\rho
\end{align*}
and the proof is finished.
\end{proof}

With Lemma \ref{lem:UB} at hand we are ready to formulate the following stability result. 

\begin{theorem}\label{thm:stability}
Let $\varepsilon_n \downarrow 0$ and let $u_0$ be a given uniform continuous function which is constant on the complement of a compact set. Let $u_n$ be the unique viscosity solution to \eqref{levelsetf} where $H_\phi$ is replaced with $H_{\phi_{\varepsilon_n}}$. Let $v_n(x,t)= u_n(x,C_{\varepsilon_n}^{-1}t)$, for any $x\in \R^d$ and $t>0$. Then $v_n\to v$ locally uniformly, where $v$ is the unique viscosity solution to \eqref{levelsetf}, where $H_\phi$ is replaced with the right-hand side of \eqref{conv-mean-nonlocal-to classical}.
\end{theorem}

\subsection{Convergence towards the spherical measure}
In this paragraph we study the case when the tails of probability measures given at \eqref{meas-lamb} concentrate at infinity. 

\begin{theorem}\label{thm3}
Let $\{\phi_{\eps}\}_{\eps>0}$ be a family of functions such that
$$
\phi_{\eps}(x)=\phi_{\eps}(-x) \quad \textnormal{and} \quad \phi_{\eps}(x) \leq j_{\eps} (x), \quad x\in\Rd,
$$
where $j_{\eps}\colon \Rd \to [0,\infty)$ are given radial and radially non-increasing functions such that for every $\eps>0$,
$$
\int_{\Rd} (1 \wedge |x|) \, j_{\eps}(x) \, \ud x < \infty.
$$
Suppose that for any $R>0$
\begin{align}\label{cond-Lambda-one}
\lim_{\varepsilon \downarrow 0}\lambda_{\phi_\varepsilon} (B_R^c) = 1.
\end{align}
Let $\eps_n \downarrow 0$ and $E_{n} \in \Reg$ be such that $E_{n} \to E$ in $\Reg$ for some $E \in \Reg$. If $0
\in \partial E \cap \partial E_{n}$ then
\begin{align*}
\lim_{n\to \infty}C_{\phi_{\eps_n}}^{-1}  H_{\phi_{\eps_n}} (E_{n})= \kappa_{d-2}^{-1}.
\end{align*}
\end{theorem}

\begin{proof}
Set $\eta = (2 \operatorname{diam} (\partial E) \vee 1)$. Without loss of generality, we can assume that $E$ is bounded. For $\eps_n$ small enough, $E, E_{n} \in B_{\eta}$ and whence
\begin{align}\label{eq:decomp-111}
\kappa_{d-2} H_{\phi_{\eps_n}} (E_{n}) &= \lim_{r\to 0^+} \int_{B_{\eta}\setminus B_r} \widetilde{\chi}_{E_{n}} (y) \phi_{\eps}(y) \ud y + \int_{B_{\eta}^c} \phi_{\eps}(y) \ud y. 
\end{align}
By our assumptions, $C_{\phi_{\eps_n}}^{-1} \int_{B_{\eta}^c} \phi_{\eps_n}(y) \ud y = \lambda_{\phi_{\eps_n}} (B_{\eta}^c) \to 1$, so we are left to show that the first term in \eqref{eq:decomp-111} (when multiplied by $C_{\phi_{\eps_n}}^{-1}$) tends to zero. Since $E_{n} \to E$ in $\Reg$, we deduce that $E_{n}$ and $E$ satisfy a uniform interior and exterior ball condition at zero, i.e.\  there exists $\delta>0$ such that $B_{\delta}(-\delta e_d) \subset E, E_{n}$ and $B_{\delta}(\delta e_d) \subset E^c, E_{n}^c$. By Lemma \ref{lem44},
\begin{align*}
\lim_{r\to 0^+} \int_{B_{\eta}\setminus B_{r}} \widetilde{\chi}_{E_{n}} (y) \phi_{\eps_n}(y) \ud y
=
\int_{B_{\eta}\setminus (B_{\delta}(-\delta e_d) \cup B_{\delta}(\delta e_d)} \widetilde{\chi}_{E_{n}} (y) \phi_{\eps_n}(y) \ud y.
\end{align*}
Further, proceeding similarly as in the proof of Lemma \ref{lem43} we obtain
\begin{align*}
C_{\phi_{\eps_n}}^{-1} \int_{B_{\eta}\setminus (B_{\delta}(-\delta e_d) \cup B_{\delta}(\delta e_d)} &\widetilde{\chi}_{E_{n}} (y) \phi_{\eps_n}(y) \ud y \\
&\leq C_{\phi_{\eps_n}}^{-1} 
\left(
\int_{B_{\eta}\setminus B_{\delta/2}} \phi_{\eps_n}(y) \ud y + c \delta \int_{B_{ \eta \wedge \frac{\delta}{2}}} (1\wedge |y|) \phi_{\eps_n}(y) \ud y \right) \\
&\leq 
\left(1 \vee \frac{2}{\delta}\right) 
\lambda_{\phi_{\eps_n}} (B_{\eta}\setminus B_{\delta/2})
+  c \delta \lambda_{\phi_{\eps_n}} (B_{ \eta \wedge \frac{\delta}{2}})\\
&\leq    \left(
\left(1 \vee \frac{2}{\delta}\right)
+ c\delta \right)\lambda_{\phi_{\eps_n}}(B_{\eta}) \to 0,\quad \mathrm{as}\ n\to \infty,
\end{align*}
and the proof is finished.
\end{proof}

\subsection{Examples}
In this short paragraph we illustrate our results by presenting a few important examples. 
\begin{example}[Asymptotics of $\alpha$-fractional curvatures as $\alpha \uparrow 1$]\label{ex1}
Corollary \ref{cor-j-profiles} allows us to recover the well-known convergence of $\alpha$-fractional curvature towards the classical curvature for sets with boundary of class $\calC^2$.
If we choose 
\begin{align}\label{rot_inv_stable_lev_dens}
j_\alpha (x) = \frac{1}{|x|^{d+\alpha}},\quad \alpha \in (0,1).
\end{align}
then 
$$
\Ku_{j_{\alpha},e} (E) =  K_{\alpha,e}(E),\quad \mathrm{and}\quad H_{j_{\alpha}} (E) =  H_{\alpha}(E).
$$
We first observe that for $C_\alpha = \int (1\wedge |x|)j_\alpha (x)\ud x$ it holds
\begin{align*}
\kappa_{d-1}^{-1} C_\alpha &= \int_1^\infty r^{-1-\alpha }\ud r + \int_0^1r^{-\alpha}\ud r = \frac{1}{\alpha}+\frac{1}{1-\alpha}\sim \frac{1}{1-\alpha},\quad \alpha \uparrow 1.
\end{align*}
Further, for any $R>0$,
\begin{align}\label{lambda-alpha-asymp}
\lambda_\alpha (B_R^c) = \kappa_{d-1} C_\alpha^{-1}\cdot
\begin{cases}
\int_R^\infty r^{-1-\alpha} \ud r = \frac{1}{\alpha}R^{-\alpha},\quad R>1;\\
\int_R^1 r^{-\alpha}\ud r + \int_1^\infty r^{-1-\alpha}\ud r= \frac{1}{1-\alpha}(1-R^{-\alpha +1})+ \frac{1}{\alpha},\quad R\leq 1
\end{cases}
\end{align}
and we easily infer that 
\begin{align*}
\lim_{\alpha\uparrow 1}\lambda_{\alpha}(B_R^c)=0,\quad R>0.
\end{align*}
Then, if we choose any $\alpha_n \uparrow 1$ and $E_n\to E$ in $\Reg$, 
 we conclude by Corollary \ref{cor-j-profiles} that
$$
\lim_{n\to \infty} C_{\alpha_n}^{-1} \Ku_{j_{\alpha_n},e} (E_n)= K_e (E),\quad \mathrm{and}\quad 
\lim_{n\to \infty} C_{\alpha_n}^{-1} H_{j_{\alpha_n}} (E_n)= H(E).
$$
This recovers \eqref{conv-fract-direct-to-class} and \eqref{conv-fract-mean-to-class} for $E_n \equiv E$, and \cite[Theorem 4.10]{Cesaroni-Stability} in the general case. 
\end{example}

\begin{example}[Asymptotics of $\alpha$-fractional curvatures as $\alpha \da 0$]
In the notation of Example \ref{ex1}, we observe that
\begin{align*}
\kappa_{d-1}^{-1} C_\alpha &= \frac{1}{\alpha}+\frac{1}{1-\alpha}\sim \frac{1}{\alpha},\quad \alpha \da 0.
\end{align*}
Thus, by \eqref{lambda-alpha-asymp}, for any $R>0$,
\begin{align*}
\lim_{\alpha\da 0}\lambda_{\alpha}(B_R^c)=1,\quad R>0.
\end{align*}
Hence, by Theorem \ref{thm3}, we obtain for $\alpha_n \downarrow 0$ and for $E_n\to E$ in $\Reg$ such that $0\in \partial E \cap \partial E_n$,
$$
\lim_{n\to \infty}  C_{\alpha_n}^{-1} \kappa_{d-2} H_{\alpha_n} (E_n)= \kappa_{d-1},
$$
which recovers 
\cite[Eq.\ (4.14)]{Cesaroni-Stability}
and
\eqref{conv-fract-at-0} by choosing $E_n\equiv E$.
\end{example}

\begin{example}\label{J_perim}
Let $j\colon \Rd \to [0,\infty)$ be a positive, radial and radially non-increasing function such that $\int_{\Rd} |x| j(x) \ud x<\infty$ and let 
$g$ be a non-negative, continuous function on $\mathbb{S}^{d-1}$.
Let $\phi(x) = j(x) g(x/|x|)$. 
For any $\varepsilon >0$ let $\phi_\varepsilon (x) = \varepsilon^{-d}\phi(x/\varepsilon)$. We note that in this case $C_{\phi_\eps} \sim \eps \int_{\Rd} |x|\phi(x) \ud x$.
We easily verify that the corresponding measures $\lambda_{\phi_\varepsilon}$ satisfy condition \eqref{tight_at_zero} and  we obtain that for $\eps_n \downarrow 0$ and for $E_n \to E$ in $\Reg$ such that $x\in \partial E \cap \partial E_n$,
\begin{equation}\label{Per_J_e}
\lim_{n\to \infty}  \eps_n^{-1} K_{\phi_{\eps_n},e} (x,E_n) =  g(e)K_{e}(x,E) \int_{0}^{\infty} r^d j(r)\ud r, \quad e \in \Ss^{d-1}(x) \cap \nu(x)^{\perp},
\end{equation}
\begin{equation}\label{Per_J_conv}
\lim_{n\to \infty}  \eps_n^{-1} H_{\phi_{\eps_n}}(x,E_n)=
\frac{\int_{0}^{\infty} r^d j(r) \ud r}{\kappa_{d-2}} \int_{\Ss^{d-1}(x) \cap \nu(x)^{\perp}} g(e) K_{e}(x,E)  \calH^{d-2}(\ud e).
\end{equation}
In particular, if 
$g\equiv 1$, $j$ is integrable and has compact support, 
and $E_n\equiv E$, then \eqref{Per_J_conv} allows us to recover \cite[Theorem 3.7]{Rossi_paper_1} \eqref{conv_j-curv-to-class}.
\end{example}

\begin{example}\label{ex10}
Let $g$ be a continuous function on $\Ss^{d-1}$ such that $g(\theta)=g(-\theta) \geq 0$ and $C_g=1$. 
Let  
$\nu$ be a locally bounded function which is regularly varying at infinity of index $-d-1$ (see e.g.\ \cite[Section 1.4.2]{Bingham}) 
such that $\int_{1}^{\infty} r^d \nu(r) \ud r = \infty$. 
Let $f_0$ be a given symmetric non-negative kernel such that $f_0 \in L^{\infty}(\Rd)$ and
$$
f_0(x) = \nu(|x|) g\left(\frac{x}{|x|}\right), \quad |x|\geq 1.
$$
For any $\eps \in (0,1)$ we set 
$$
f_{\eps}(x) = \frac{1}{\eps^d} f_0\left(\frac{x}{\eps}\right).
$$
At this place it is worth to observe that we cannot apply
Theorem \ref{thm2} directly to the kernels $\{f_{\eps}\}_{\eps>0}$. We therefore introduce a family of auxiliary kernels which enable us to exploit Proposition \ref{prop_nu}.
For that reason, we set $\tilde{\nu}(r) := \sup \{\nu(s): s \geq r\}$ and
$$
\phi(x) = g\left(\frac{x}{|x|}\right) j(x), \quad \textnormal{where} \quad j(x) = \tilde{\nu}(1) \wedge\tilde{\nu}(|x|).
$$
Further, for any $\eps \in (0,1)$, let
$$
\phi_{\eps}(x) = \frac{1}{\eps^d} \phi\left(\frac{x}{\eps}\right).
$$
It evidently follows that
$$
f_{\eps}(x) = \phi_{\eps}(x) + \overline{\phi}_{\eps} (x),
$$
where
$$
\overline{\phi}_{\eps} (x) = 
\frac{1}{\eps^d}
\left(
\left(f_0\left(\frac{x}{\eps}\right) -  g\left(\frac{x}{|x|}\right)\tilde{\nu}(1)\right) \mathds{1}_{|x|<\eps}
+\left(\nu\left(\frac{|x|}{\eps}\right) - \tilde{\nu}\left(\frac{|x|}{\eps}\right) \right)g\left(\frac{x}{|x|}\right) \mathds{1}_{|x|\geq\eps}\right). 
$$
By \cite[Theorem 1.5.3]{Bingham}, $\tilde{\nu}(r) \sim \nu(r)$ as $r \to\infty$. Since $\tilde{\nu}$ is regularly varying of index $-d-1$, it has the following representation $\tilde{\nu}(r) = r^{-d-1} \ell(r)$, where $\ell$ is slowly varying at infinity. 
We observe that
\begin{align*}
C_{\phi_{\eps}} &=
\eps \frac{\ell(1)}{d+1} + \eps \int_{1}^{1/\eps} \ell(r) r^{-1} \ud r
+ \int_{1/\eps}^{\infty} \ell (r) r^{-2} \ud r
\geq \eps  \int_{1}^{1/\eps} \ell(r) r^{-1} \ud r.
\end{align*}
We set
$$
\overline{L}(r) = \int_1^r \ell(s) s^{-1} \ud s, \quad r\geq 1.
$$
By \cite[Proposition 1.5.8]{Bingham}, $\ell(r)/\overline{L}(r) \to 0$ as $r\to \infty$.
By \cite[Proposition 1.5.10]{Bingham}, for any $R>0$
\begin{align*}
\lambda_{\phi_{\eps}}(B_R^c) \leq
\frac{\int_{R/\eps}^{\infty}  r^{-2} \ell(r) \ud r}{\eps  \overline{L}({1/\eps})}  
&\sim\frac{\ell(R/\eps)}{R  \overline{L}(1/\eps)} 
\sim \frac{\ell(1/\eps)}{R  \overline{L}(1/\eps)} \stackrel{\eps\da 0}{\longrightarrow} 0.
\end{align*}
We conclude that Theorem \ref{thm2} is valid for the family of kernels $\{\phi_{\eps}\}_{\eps>0}$. In the next step we aim
to verify that the family $\{\overline{\phi}_{\eps}\}_{\eps>0}$ satisfies the assumptions of Proposition \ref{prop_nu}.
For that reason, we introduce a function $F$ such that 
$$
\overline{L}(F(\eps)) =  \big(\overline{L}(1/\eps)\big)^{1/2}.
$$
We observe that $F(\eps)$ tends to infinity as $\eps\to0$.
We next set
$$
\delta(\eps) = \sup_{r\geq F(\eps)} \left|\frac{\nu(r)}{\tilde{\nu}(r)} - 1 \right|.
$$
Evidently, $\delta(\eps) \to 0$ as $\eps \to 0$.
Let
$$
\overline{j}_{\eps}(x) =\frac{1}{\eps^d} \left((c_1+c_2\tilde{\nu}(1))\mathds{1}_{|x|<\eps}
+ c_2  \tilde{\nu}\left(\frac{|x|}{\eps}\right)  \mathds{1}_{\eps\leq|x|\leq F(\eps)\eps} 
+ c_2 \delta(\eps) \tilde{\nu}\left(\frac{|x|}{\eps}\right) \mathds{1}_{|x|\geq F(\eps) \eps}\right), 
$$
where $c_1 = \mathrm{ess \, sup} f_0, c_2 = \sup g$. Observe that
$|\overline{\phi}_{\eps} (x)| \leq \overline{j}_{\eps}(x) \, \mathrm{a.e.}$ and $\int_{\Rd} (1\wedge |x|) \overline{j}_{\eps}(x) \ud x<\infty$.
We conclude that 
$$
\int_{\Rd}(1\wedge |x|) \overline{j}_{\eps} (x) \ud x \sim \eps \kappa_{d-1} \left(
\frac{c_1+c_2\tilde{\nu}(1)}{d+1} + c_2  \overline{L}(F(\eps)) + c_2\delta(\eps) \int_{F(\eps)}^{1/\eps}\ell(r) r^{-1} \ud r + c_2 \delta (\eps)\ell(1/\eps)
\right)
$$
which yields
\begin{align*}
\lim_{\eps \downarrow 0} C_{\phi_{\eps}}^{-1} \int_{\Rd} (1\wedge |x| ) \overline{j}_{\eps} (x) \ud x = 0.
\end{align*}
In the light of \cite[Proposition 1.5.8 and Proposition 1.5.9a]{JEMS}, we have 
\begin{equation}\label{eikonal}
C_{\phi_{\eps}} \sim \eps \overline{L}(1/\eps), \quad \textnormal{as} \,\, \eps \downarrow 0. 
\end{equation}
For any $\eps>0$ we set
$$
\psi_{\eps} (x) = \frac{f_{\eps}(x)}{\eps \overline{L}(1/\eps)}, \quad x\in \Rd.
$$
Finally, by Proposition \ref{prop_nu}, for any $\varepsilon_n \downarrow 0$ and $E_n\to E$ in $\Reg$,
$$
\lim_{n\to\infty} H_{\psi_{\eps_n}}(x,E_n) =
 \frac{1}{\kappa_{d-2}} \int_{\Ss^{d-2}} g(e) K_e (x,E) \calH^{d-2} (\ud e).
$$
If we choose $\nu(r) = \frac{1}{r^{d+1}}$, then equation \eqref{eikonal} becomes $C_{\phi_{\eps}} \sim \eps \log(1/\eps)$, as $\eps \downarrow 0$, and
\begin{equation}\label{psi-jems}
\psi_{\eps}(x) = \frac{1}{\eps^{d+1} \log (1/\eps)} f_0\left(\frac{x}{\eps}\right),
\end{equation}
and this choice allows us to relate the present example to the {\em dislocation model} studied in \cite{JEMS}. More precisely, \eqref{psi-jems} coincides with the scaling from \cite[Eq.\ (1.3)]{JEMS}.
\end{example}

\begin{example}
Let $g$ be a continuous function on $\Ss^{d-1}$ such that $g(\theta)=g(-\theta) \geq 0$ and $C_g=1$. 
Let $\nu$ be a locally bounded function which is regularly varying at infinity of index $-d$ 
and such that $\int_1^{\infty} r^{d-1} \nu(r) \ud r <\infty$.
Let $f_0$ be a a given symmetric non-negative kernel such that $f_0 \in L^{\infty}(\Rd)$ and
$$
f_0(x) = \nu(|x|) g\left(\frac{x}{|x|}\right), \quad |x|\geq 1.
$$
For any $\eps \in (0,1)$ we set 
$$
f_{\eps}(x) = \frac{1}{\eps^d} f_0\left(\frac{x}{\eps}\right).
$$ 
Let $\tilde{\nu}(r) := \sup \{\nu(s): s \geq r\}$. By \cite[Theorem 1.5.3]{Bingham}, $\tilde{\nu}(r) \sim \nu(r)$ as $r \to\infty$.
We further set 
$$
j_{\eps} (x) = \frac{1}{\eps^d}\left(c_1 \mathds{1}_{|x| < \eps} + c_2 \tilde{\nu}\left(\frac{|x|}{\eps}\right)\mathds{1}_{|x|>\eps}\right),
$$
where $c_1 = \mathrm{ess \, sup} f_0, c_2 = \sup g$. Observe that $f_{\eps}(x) \leq j_{\eps}(x)$ and $\int_{\Rd} (1\wedge |x|) j_{\eps}(x) \ud x <\infty$.
Since $\nu$ is regularly varying of index $-d$, 
it has the following representation
$\nu(r) = r^{-d} \ell(r)$, where $\ell$ is slowly varying at infinity. 
We set
$$
\underline{L}(r) = \int_r^{\infty} \ell(s) s^{-1} \ud s, \quad r\geq 0.
$$
By \cite[Proposition 1.5.8]{Bingham}, 
\begin{align*}
C_{f_{\eps}} &=
\eps \int_{B_1}|x|f_0(x) \ud x + \eps \int_{1}^{1/\eps} \ell(r) \ud r
+ \underline{L}(1/\eps) \\
&\sim \eps \int_{B_1}|x|f_0(x) \ud x + \ell(1/\eps) + \underline{L}(1/\eps).
\end{align*}
By \cite[Proposition 1.5.9b]{Bingham}, for any $R>0$
\begin{align*}
\lambda_{f_{\eps}}(B_R^c) \geq
\frac{1}{C_{f_{\eps}}} 
\underline{L}((1\vee R)/\eps)
&\sim \frac{\underline{L}((1\vee R)/\eps) 
}{\eps \int_{B_1}|x|f_0(x) \ud x + \ell(1/\eps) + \underline{L}(1/\eps)
}  \stackrel{\eps\da 0}{\longrightarrow} 1.
\end{align*}
Hence, by Theorem \ref{thm3}, for any $\varepsilon_n \downarrow 0$ and $E_n\to E$ in $\Reg$,
$$
\lim_{n\to\infty} C_{f_{\eps_n}}^{-1} H_{f_{\eps_n}}(x,E_n) = \kappa_{d-2}^{-1}.
$$
\end{example}

\subsection{Asymptotics of anisotropic fractional curvatures}

The notion of non-local curvature introduced in the present article can be also applied in the framework of anisotropic curvatures as we shall show that such curvatures can be seen as the first variation of the corresponding nonolocal anisotropic  perimeters. 
Our first  task is to study the limit behaviour of \eqref{anisotr-fract-curv} in the case when $\alpha \uparrow 1$. 
 We start by identifying  the limit  and next we prove that it is the first variation of the anisotropic perimeter  $\Per(E,\mathcal{ZK})$.
\begin{theorem}\label{thm:conv-anisotropic-curv}
Let $\alpha_n \uparrow 1$ and $E_{n} \in \Reg$ be such that $E_{n} \to E$ in $\Reg$ for some $E$ in $\Reg$. Then, for every $x\in\partial E \cap \partial E_{n}$,
it holds
\begin{align}\label{Anisotr-limit}
\lim_{n\to \infty}(1-\alpha_n) H_{\alpha_n}^\mathcal{K} (x,E_n) =
\frac{1}{\kappa_{d-2}} \int_{e\in \Ss^{d-1}(x)\cap \, \nu(x)^{\perp}} \frac{K_{e}(x,E)}{\|e\|_\mathcal{K}^{d+1}}\, \calH^{d-2} (\ud e).
\end{align}
\end{theorem}

\begin{proof}
Without loss of generality we can again work in normal coordinates introduced in Section \ref{sec:Prem}.
We consider the following kernel
\begin{align*}
\phi_{\alpha}^\mathcal{K}(x) = g_{\alpha}^{\mathcal{K}}(x/|x|) j_{\alpha} (x),
\end{align*}
where $g_{\alpha}^\mathcal{K} (\theta) = \|\theta\|_\mathcal{K}^{-d-\alpha}$ and $j_{\alpha}(x) = |x|^{-d-\alpha}$. 
With this choice we have
\begin{align}
H_{\phi_\alpha^\mathcal{K}}(E)= H_\alpha^\mathcal{K} (E).
\end{align} 
We compute that
$$
C_{\alpha}^\mathcal{K} 
=
\int (1\wedge |x|)\phi_\alpha^\mathcal{K} (x)\ud x 
= C_{\alpha}(\mathcal{K}) \left(\frac{1}{\alpha} + \frac{1}{1-\alpha}\right),
$$
where 
$$
C_\alpha (\mathcal{K}) = \int_{\Ss^{d-1}}\frac{\calH^{d-1}(\ud \theta)}{\Vert \theta \Vert_\mathcal{K}^{d+\alpha}}\to \int_{\Ss^{d-1}}\frac{\calH^{d-1}(\ud \theta)}{\Vert \theta \Vert_\mathcal{K}^{d+1}},\quad \mathrm{as}\ \alpha \uparrow 1.
$$
We also easily find that the measures $\lambda_\alpha^\mathcal{K}$ appearing in Theorem \ref{thm2} satisfy
\begin{align*}
\lambda_\alpha^\mathcal{K}(B_R^c) = \frac{\int_R^\infty (1\wedge r)r^{-\alpha -1}\ud r}{\int_0^\infty (1\wedge r)r^{-\alpha -1}\ud r}
\end{align*}
and we show similarly as in Example \ref{ex1} that $\lambda ^\mathcal{K}_\alpha(B_R^c)$
converges to $0$ as $\alpha \uparrow 1$.
Further, Dini's theorem implies that $\lim_{\alpha \uparrow 1}g_\alpha^\mathcal{K}(\theta) = \Vert \theta \Vert_\mathcal{K}^{-d-1}$ uniformly in $\theta \in \Ss^{d-1}$.
We conclude by Theorem \ref{thm2} that if $\alpha_n \uparrow 1$ then
$$
\lim_{n\to \infty} (1-\alpha_n) K_{\phi_{\alpha_n}^K,e} (E_n)
= \frac{K_e(E)}{\|e\|_{\mathcal{K}}^{d+1} }, \quad e\in\Ss^{d-2},
$$
and
$$
\lim_{n\to \infty} (1-\alpha_n) H_{\alpha_n}^\mathcal{K} (E_n) = \frac{1}{\kappa_{d-2}} \int_{\Ss^{d-2}} \frac{K_e(E)}{\|e\|_{\mathcal{K}}^{d+1}} \, \calH^{d-2}(\ud e),
$$
which finishes the proof.
\end{proof}

We close the paper with the result concerning the first variation of the anisotropic perimeter.
\begin{theorem}\label{anisotr-thm-first-var}
The quantity appearing on the right-hand side of \eqref{Anisotr-limit} is the first variation of the anisotropic perimeter $\Per (E,\mathcal{ZK})$.
\end{theorem}

\begin{proof}
By \cite[Theorem 3.7]{Bellettini} we know that $\Per (E,\mathcal{ZK})$ is the Minkowski content related to the distance $d_\mathfrak{N}(x,y)=\fN(x-y)$, where $\fN(x) = \Vert x \Vert_{\mathcal{ZK}}$ is a convex norm on $\R^d$ linked with the convex body $\mathcal{ZK}$. The dual norm $\fN^0(x) = \Vert x\Vert_{\mathcal{Z}^\ast \mathcal{K}}$ is linked with the polar body $\mathcal{Z}^\ast \mathcal{K}$.
 By \cite[Theorem 4.3]{Bellettini}, the first variation of $\Per (E,\mathcal{ZK})$ is given by the following curvature
\begin{align*}
\kappa_\fN^E (x) = \mathrm{div}(\mathsf{n}_\fN^{E}(x)),
\end{align*}
where $\mathsf{n}_\fN$ is the Cahn-Hoffman vector field defined as
\begin{align*}
\mathsf{n}_\fN^{E}(x) = \nabla \fN^0 (\nabla d_\fN^{E}(x)),\quad x\in U,
\end{align*}
where $U$ is a neighbourhood of $\partial E$ and $d_\fN^{E}$ is the oriented $\fN$-distance function given by
\begin{align*}
d_\fN^{E}(x) = \mathrm{dist}_\fN (x,E) - \mathrm{dist}_\fN(x,E^c),\quad x\in \R^d,
\end{align*}
which is smooth in $U$ and satisfies the anisotropic eikonal equation, that is
\begin{align*}
\fN^0 (\nabla d_\fN^{E}(x)) =1, \quad x\in U.
\end{align*}
We easily calculate 
that the right-hand side of \eqref{Anisotr-limit} is equal  (up to the constant $\kappa_{d-2}$) to
$\mathrm{tr}\!\left( A(\nu^{E}(x)D^2d^{E}(x)\right)$, for $x\in \partial E$, where 
\begin{align*}
A(p/|p|) = \frac{1}{2}\int_{\Ss^{d-1}\cap (p/|p|)^{\perp}}g(\theta)(\theta \otimes \theta ) \calH^{d-2}(\ud \theta),\qquad 
g(\theta) = \Vert \theta \Vert_{K}^{-d-1},
\end{align*}
and $d^{E}$ is the classical oriented distance defined as
\begin{align*}
d^{E}(x) = \mathrm{dist} (x,E) - \mathrm{dist} (x,E^c),\quad x\in \R^d,
\end{align*}
which is smooth in $U$ and satisfies the classical eikonal equation $|\nabla d^{E}(x)|=1$, for $x \in U$. Here, $v^{E}(x)$ is the outer unit normal vector at  $x \in \partial E$ and it is known that $v^{E}(x) = \nabla d^{E}(x)$. Moreover, the Hessian matrix of $d^{E}$ is equal to the Weingarten map for $\partial E$, see \cite{Giusti}.
We therefore conclude that the result will follow if we show that 
\begin{align}\label{to-show-1}
\kappa_\fN^E(x)
= c_d
\mathrm{tr}\!\left( A(\nu^{E}(x)D^2d^{E}(x)\right),
\quad x\in \partial E,
\end{align}
for a constant $c_d>0$ depending only on dimension $d$. 
By \cite[Theorem 1.7]{JEMS}, the right hand-side of \eqref{to-show-1} is equal to 
$\mathrm{div} \! \left( \nabla G \big(\nabla d^{E}(x)\big) \right)$, where $G=-(2\pi)^{-1}\mathcal{F}(L_g)$ and $L_g$ is the following tempered distribution related to the function $g$, 
\begin{align*}
\langle L_g, \vp \rangle = \int g(x) \left( \vp (x) -\vp(0) -x\cdot \nabla \vp (0) \mathbf{1}_{B_1}(x) \right)\ud x
\end{align*}
and $\mathcal{F}(\vp) (\xi) = \int \vp(x) e^{-i\xi\cdot x}\ud x$ is the Fourier transform of a test function $\vp$ coming from the Schwartz space $ \mathcal{S}(\R^d)$. We first find the function $G$. We have
\begin{align*}
\langle \mathcal{F}(L_g),\vp \rangle 
&=
\langle L_g, \mathcal{F}(\vp) \rangle
=
 \int g(\xi) \left( \mathcal{F}\vp (\xi) -\mathcal{F}(\vp)(0) -\xi\cdot \nabla \mathcal{F}(\vp) (0) \mathbf{1}_{B_1}(\xi) \right)\ud \xi\\
&=
 \int g(\xi) \int \vp (x) \left( e^{-i-\xi\cdot x}-1+i\xi \cdot x \mathbf{1}_{B_1}(\xi)\right) \ud x\, \ud \xi \\
 &=
 \int \vp (x)  \int_{\Ss^{d-1}}\int_0^\infty \Vert \theta \Vert_\mathcal{K}^{-d-1}\left( e^{-ir\theta \cdot x}-1+ir\theta \cdot x \mathbf{1}_{(0,1)}(r)\right)r^{-2} \ud r \calH^{d-1}(\ud \theta) \ud x\\
 &=
 -\int \vp(x) \int_{\Ss^{d-1}}|\theta \cdot x|\frac{\calH^{d-1}(\ud \theta)}{\Vert \theta \Vert_\mathcal{K}^{d+1}}\, \ud x 
 =
 -(d+1)
 \int \vp(x) \int_K |y\cdot x|\ud y\, \ud x\\
 &=
 -\langle 2\fN^{0}, \vp \rangle, 
\end{align*}
where the penultimate equality follows by \cite[Eq.\ (3.14)]{Cygan-Grzywny-Per} and  the last line is just \eqref{moment_body}. We thus obtain that
\begin{align*}
G(x) = \frac{1}{\pi} \fN^0(x).
\end{align*}
We claim that  the gradient of $G$ is given by
\begin{align}\label{grad-G}
\nabla G(x) = \frac{1}{\pi} \frac{x}{|x|^2}\fN^0(x).
\end{align}
To show the claim it suffices to prove that for $F(x) = \int_\mathcal{K} |x\cdot y|\ud y$ it holds $\nabla F(x) = \frac{x}{|x|^2}F(x)$. We clearly have
\begin{align*}
\lim_{t\to 0}\frac{\left\vert \left( x+t\frac{x}{|x|}\right)\cdot y\right\vert - |x\cdot y|}{t} = \frac{x}{|x|}\cdot y
\end{align*}
and this yields 
\begin{align}\label{grad-F-1}
\nabla F(x) \cdot \frac{x}{|x|} = \frac{F(x)}{|x|}.
\end{align}
We next show that 
\begin{align}\label{grad-F-2}
\nabla F(x) \cdot \frac{x^\prime}{|x^\prime|}=0,
\end{align}
for any $x^\prime\in x^\perp=\{ z\in \R^d: x\cdot z=0\}$. Indeed, we decompose any $y\in \R^d$ as follows $y=sx+y^\prime$, where $y^\prime \in x^\perp$ and we observe that for $t$ small enough we have
\begin{align*}
\left\vert \left( x+t\frac{x^\prime}{|x^\prime|}\right)\cdot y\right\vert - |x\cdot y|
&=
\left\vert s|x|^2\ + \frac{t}{|x^\prime|}(x^\prime \cdot y^\prime)  \right\vert - |s|\!\cdot \!|x|^2\\
&=
\begin{cases}
\frac{t}{|x^\prime|}(x^\prime \cdot y^\prime),& s>0\\
-\frac{t}{|x	^\prime|}(x^\prime \cdot y^\prime),& s<0\\
\frac{t}{|x^\prime|}|(x^\prime \cdot y^\prime)|,& s=0.
\end{cases}
\end{align*}
Since the set $\mathcal{K}$ is origin symmetric, the two integrals (for $s>0$ and $s<0$) appearing in the left hand-side of \eqref{grad-F-2}  will cancel out. The last case ($s=0$) is irrelevant as then $y\in x^\perp$, but at the same time  $F(x)=\int_{\mathcal{K}\, \cap \,(x^\perp )^c}|x\cdot y|\ud y$. Finally, \eqref{grad-F-1} and \eqref{grad-F-2} evidently imply \eqref{grad-G}. 
This yields
\begin{align*}
\mathrm{tr}\!\left( A(\nu^{E}(x)D^2d^{E}(x)\right)
&=
\mathrm{div} \! \left( \nabla G \big(\nabla d^{E}(x)\big) \right)\\
&=
\mathrm{div} \! \left( v^{E}(x) \fN^0 (v^{E}(x)) \right)
=
\mathrm{div} \! \left( v_\fN^{E}(x) (\fN^0 (v^{E}(x)))^2  \right),
\end{align*}
where the last equality follows by \cite[Eq.\ (4-5)]{Bellettini}
and $v_\fN^{E}(x) = \frac{v^{E}(x)}{\fN^0(v^{E}(x))}$.
At last, we deal with the left hand-side of \eqref{to-show-1}.  By \cite[Eq.\ (4-4) and (4-7)]{Bellettini},
\begin{align*}
\kappa_\fN^E (x)
&=
\mathrm{div} (\nabla \fN^0 (\nabla d_\fN^{E}(x))) =\mathrm{div} (\nabla \fN^0 (v_\fN^{E}(x)))\\
&=
\mathrm{div} \left( \frac{v_\fN^{E}(x)}{|v_\fN^{E}(x)|^2} \fN^0 (v_\fN^{E}(x)) \right)
=
\mathrm{div} \left( v_\fN^{E}(x) \big( \fN^0 (v_\fN^{E}(x)) \big)^2  \right)
\end{align*}
and the proof is complete.
\end{proof}

\begin{remark}
Our stability Theorem \ref{thm:stability} can be also applied in the present framework of anisotropic curvatures. The limit curvature appearing on the right-hand side of \eqref{Anisotr-limit} enjoys property (C') and thus locally unifrom convergence of viscosity solutions holds. 
\end{remark}

\subsection*{Acknowledgement.}
We would like to thank Micha{\l} Kijaczko (Wroc{\l}aw University of Science and Technology) for helpful discussions.

\bibliographystyle{abbrv}

\end{document}